\documentclass[leqno,11pt,a4paper]{amsart}
\usepackage[full]{textcomp}
\usepackage{newpxtext}
\usepackage{cabin} 
\usepackage[varqu,varl]{inconsolata} 
\usepackage[bigdelims,vvarbb]{newpxmath} 
\usepackage[cal=cm,bb=esstix,bbscaled=1.05]{mathalfa} 
\usepackage[margin=2.4cm]{geometry}
\usepackage{savesym}

\usepackage{mathabx}
\usepackage[usenames,dvipsnames]{color}
\usepackage{hyperref}
\usepackage{tikz}
\usetikzlibrary{cd}
\usepackage{graphicx}
\usepackage[all,cmtip]{xy}
\usepackage{mleftright}
\usepackage{booktabs}
\usepackage{cite}
\usepackage{enumitem}
\usepackage{mathdots}
\usepackage{eucal}

\setlist[enumerate]{labelsep=*, leftmargin=1.5pc}
\setlist[enumerate]{label=\normalfont(\roman*), ref=\roman*}
\newtheorem{thm}{Theorem}[section]
\newtheorem{lemma}[thm]{Lemma}
\newtheorem{cor}[thm]{Corollary}
\newtheorem{prop}[thm]{Proposition}

\newtheorem{conjecture}[thm]{Conjecture}
\theoremstyle{definition}
\newtheorem{example}[thm]{Example}
\newtheorem{remark}[thm]{Remark}
\newtheorem{definition}[thm]{Definition}
\newtheorem{notation}[thm]{Notation}
\newtheorem{question}[thm]{Question}
\newtheorem{problem}[thm]{Problem}
\newtheorem{convention}[thm]{Convention}
\newtheorem{setup}[thm]{Setup} 
\newtheorem{construction}[thm]{Construction} 
\numberwithin{equation}{section}
\newcommand{\Z}{\mathbb{Z}}

\newcommand{\R}{\mathbb{R}}
\newcommand{\C}{\mathbb{C}}
\newcommand{\D}{\mathbb{D}}

\newcommand{\vol}[1]{\operatorname{vol}\mleft({#1}\mright)}

\renewcommand{\dim}[1]{\operatorname{dim}\mleft({#1}\mright)}

\renewcommand{\min}[1]{\operatorname{min}\mleft\{{#1}\mright\}}
\renewcommand{\max}[1]{\operatorname{max}\mleft\{{#1}\mright\}}

\newcommand{\on}{\operatorname}

\newcommand{\mat}{\left(\begin{array}}
\newcommand{\tam}{\end{array}\right)}

\newcommand{\Sp}{\on{Sp}(2)}
\newcommand{\Spx}{\on{Sp}_\star(2)}
\newcommand{\twSp}{\widetilde{\on{Sp}}(2)}
\newcommand{\twSpx}{\widetilde{\on{Sp}}_\star(2)}
\newcommand{\twU}{\widetilde{\on{U}}(1)}
\newcommand{\rot}{\on{rot}}
\newcommand{\Ru}{\on{Ru}}
\newcommand{\ru}{\on{ru}}
\newcommand{\dvol}{\on{dvol}}
\newcommand{\area}{\on{area}}
\newcommand{\diam}{\on{diam}}
\newcommand{\sys}{\on{sys}}
\newcommand{\CZ}{\on{CZ}}
\newcommand{\Cal}{\on{Cal}}

\definecolor{dark green}{rgb}{0.0, 0.6, 0.0}

\graphicspath{{images/}}
\begin{document}
\author[J.~Chaidez]{J.~Chaidez}
\address{Department of Mathematics\\University of California at Berkeley\\Berkeley, CA\\94720\\USA}
\email{jchaidez@berkeley.edu}
\author[O.~Edtmair]{O.~Edtmair}
\address{Department of Mathematics\\University of California at Berkeley\\Berkeley, CA\\94720\\USA}
\email{oliver\_edtmair@berkeley.edu}
\title[3D Convex Contact Forms And The Ruelle Invariant]{3D Convex Contact Forms And The Ruelle Invariant}

\begin{abstract} Let $X \subset \R^4$ be a convex domain with smooth boundary $Y$. We use a relation between the extrinsic curvature of $Y$ and the Ruelle invariant of the Reeb flow on $Y$ to prove that there are constants $C > c > 0$ independent of $Y$ such that
\[c \le \ru(Y) \cdot \sys(Y)^{1/2} \le C\]
Here $\sys(Y)$ is the systolic ratio of $Y$, i.e. the square of the minimal period of a closed Reeb orbit of $Y$ divided by twice the volume of $X$, and $\ru(Y)$ is the volume-normalized Ruelle invariant. We then construct dynamically convex contact forms on $S^3$ that violate this bound using methods of Abbondandolo-Bramham-Hryniewicz-Salom\~{a}o. These are the first examples of dynamically convex contact $3$-spheres that are not strictly contactomorphic to a convex boundary $Y$.
\end{abstract}

\maketitle

\section{Introduction} \label{sec:introduction} A contact manifold $(Y,\xi)$ is an odd dimensional manifold equipped with a hyperplane field $\xi \subset TY$, called the contact structure, that is the kernel of a $1$-form $\alpha$ such that
\[\ker(d\alpha) \subset TY \text{ is rank 1} \qquad\text{and}\qquad \alpha|_{\ker(d\alpha)} > 0\]
A $1$-form satisfying this condition is called a contact form on $(Y,\xi)$. Every contact form comes equipped with a natural Reeb vector field $R$, defined by
\[\alpha(R) = 1 \qquad \iota_Rd\alpha = 0\]
Note that the Reeb vector-field preserves the $1$-form $\alpha$ and the natural volume form $\alpha \wedge d\alpha^{n-1}$, where $\dim{Y} = 2n - 1$. The dynamical properties of Reeb vector fields (e.g. the existence of closed orbits and their properties) are the subject of immense interest in symplectic geometry and dynamical systems.

\vspace{3pt}

Contact manifolds arise naturally as hypersurfaces in symplectic manifolds satisfying a certain stability condition. In fact, Weinstein introduced contact manifolds in \cite{w1979} inspired by the following prototypical example of this phenomenon, due to Rabinowitz \cite{r1979}. 

\begin{example} We say that a domain $X \subset \R^{2n}$ with smooth boundary $Y$ is \emph{star-shaped} if
\[0 \in \on{int}(X) \qquad\text{and}\qquad \partial_r \text{ is transverse to }Y\]
Let $\omega$ and $Z$ denote the standard symplectic form and Liouville vector field on $\R^{2n}$. That is
\begin{equation} \label{eq:standard_omega_Z} \omega = \sum_{i=1}^n dx_i \wedge dy_i \qquad Z = \frac{1}{2} \sum_i x_i \partial_{x_i} + y_i \partial_{y_i} = \frac{1}{2}r\partial_r \end{equation}
Then the restriction $\lambda|_Y$ of the Liouville $1$-form $\lambda = \iota_Z\omega$ is a contact form.\end{example}

\begin{example} The \emph{standard} contact structure $\xi$ on $S^{2n-1} \subset \R^{2n}$ is given by $\xi = \ker(\lambda|_{S^{2n-1}})$.
\end{example}

\noindent Every contact form on the standard contact sphere arises as the pullback of $\lambda$ via a map to a star-shaped boundary $Y$. Indeed, if $\alpha = f \cdot \lambda|_{S^{2n-1}}$ for $f > 0$ is a contact form for $\xi$, then
\[
\alpha = \phi^*\lambda \qquad\text{where}\qquad \phi(\theta) = (f(\theta)^{1/2},\theta) \quad\text{in radial coordinates }(0,\infty)_r \times S^{2n-1}_\theta
\]
Moreover, every star-shaped boundary $Y$ admits such a map from the sphere. Thus, from the perspective of contact geometry, the study of star-shaped boundaries is equivalent to the study of contact forms on the standard contact sphere.

\subsection{Convexity} \label{subsec:convexity} In this paper, we are primarily interested in studying contact forms arising as boundaries of convex domains.

\begin{definition} \label{def:convex_contact_forms} A contact form $\alpha$ on $S^{2n-1}$ is \emph{convex} if there is a convex star-shaped domain $X \subset \R^{2n}$ with boundary $Y$ and a strict contactomorphism $(S^3,\alpha) \simeq (Y,\lambda|_Y)$. \end{definition}

\noindent In contrast to the star-shaped case, not every contact form on $S^{2n-1}$ is convex, and the Reeb flows of convex contact forms possess many special dynamical properties, both proven and conjectural.

\vspace{3pt}

In \cite{v2000}, Viterbo proposed a particularly remarkable systolic inequality for Reeb flows on convex boundaries. To state it, let $(Y,\alpha)$ be a closed contact manifold with contact form of dimension $2n - 1$, and recall that the volume $\vol{Y,\alpha}$ and systolic ratio $\sys(Y,\alpha)$ are given by

\begin{equation}
\vol{Y,\alpha} = \int_Y \alpha \wedge d\alpha^{n-1} \qquad\text{and}\qquad \sys(Y,\alpha) = \frac{\min{\text{period $T$ of an orbit}}^n}{\vol{Y,\alpha}}
\end{equation}

\noindent Note that if $Y$ is the boundary of a star-shaped domain $X\subset\R^{2n}$, then the contact volume of $(Y,\lambda|_Y)$ is related to the volume of $X$ via $\vol{Y,\lambda|_Y}= n!\vol{X}$. The weak Viterbo conjecture that originally appeared in \cite{v2000} can be stated as follows.

\begin{conjecture} \cite{v2000} Let $\alpha$ be a convex contact form on $S^{2n-1}$. Then the systolic ratio is bounded by $1$.
\[\sys(S^{2n-1},\alpha) \le 1\]
\end{conjecture}

\noindent There is also a strong Viterbo conjecture (c.f. \cite{ghr2020}), stating that all normalized symplectic capacities are equal on convex domains. For other special properties of convex domains, see \cite{hwz1998,v2000}.

\vspace{3pt}

Despite the plethora of distinctive properties that convex contact forms possess, a characterization of convexity entirely in terms of contact geometry has remained elusive.  

\begin{problem} \label{problem:what_is_convexity} Give an intrinsic characterization of convexity that does not reference a map to $\R^{2n}$.\end{problem}

\subsection{Dynamical Convexity} In the seminal paper \cite{hwz1998}, Hofer-Wysocki-Zehnder provided a candidate answer to Problem \ref{problem:what_is_convexity}. 

\begin{definition}[Def. 3.6, \cite{hwz1998}] \label{def:dynamically_convex} A contact form $\alpha$ on $S^3$ is \emph{dynamically convex} if the Conley-Zehnder index $\CZ(\gamma)$ of any closed Reeb orbit $\gamma$ is greater than or equal to $3$. \end{definition}

The Conley-Zehnder index of a Reeb orbit plays the role of the Morse index in symplectic field theory and other types of Floer homology (see \S \ref{subsec:CZ_index} for a review). Thus, on a naive level, dynamical convexity may be viewed as a type of ``Floer-theoretic'' convexity. If $X$ is a convex domain whose boundary $Y$ has positive definite second fundamental form, then $Y$ is dynamically convex \cite[Thm 3.7]{hwz1998}. Note that this condition is open and dense among convex boundaries.

\vspace{3pt}

In \cite{hwz1998}, Hofer-Wysocki-Zehnder proved that the Reeb flow of a dynamically convex contact form admits a surface of section. In the decades since, dynamical convexity has been used as a key hypothesis in many significant works on Reeb dynamics and other topics in contact and symplectic geometry. See the papers of Hryniewicz \cite{h2014}, Zhou \cite {z2019a,z2019b}, Abreu-Macarini \cite {am2014,am2015}, Ginzburg-G{\"u}rel \cite{gg2016}, Fraunfelder-Van Koert \cite{fvk2018} and Hutchings-Nelson \cite{hn2014} for just a few examples. However, the following question has remained stubbornly open (c.f. \cite[p. 5]{fvk2018}).

\begin{question} \label{qu:dynamical_convex_is_convex} Is every dynamically convex contact form on $S^3$ also convex?
\end{question} 

The recent paper \cite{abhs2} of Abbondandolo-Bramham-Hryniewicz-Salom\~{a}o (ABHS) has suggested that the answer to Question \ref{qu:dynamical_convex_is_convex} should be no. They construct dynamically convex contact forms on $S^3$ with systolic ratio close to $2$. There is substantial evidence for the weak Viterbo conjecture (cf. \cite{ch2020}), and so these contact forms are likely \emph{not} convex. However, this was not proven in \cite{abhs2}.

\vspace{3pt}

Even more recently, Ginzburg-Macarini \cite{gm2020} addressed a version of Question \ref{qu:dynamical_convex_is_convex} in higher dimensions that incorporates the assumption of symmetry under the antipod map $S^{2n-1} \to S^{2n-1}$. Their work did not address the general case of Question \ref{qu:dynamical_convex_is_convex}.

\subsection{Main Result} The main purpose of this paper is to resolve Question \ref{qu:dynamical_convex_is_convex}.

\begin{thm} \label{thm:dynamical_convex_is_not_convex} There exist dynamically convex contact forms $\alpha$ on $S^3$ that are not convex. \end{thm}
\noindent Theorem \ref{thm:dynamical_convex_is_not_convex} is an immediate application of Proposition \ref{prop:main_prop_1} and \ref{prop:main_prop_2}, which we will now describe.

\vspace{3pt}

\subsection{Ruelle Bound} For our first result, recall that any closed contact $3$-manifold $(Y,\xi)$ with contact form $\alpha$ that satisfies $c_1(\xi) = 0$ and $H^1(Y;\Z) = 0$ has an associated \emph{Ruelle invariant} \cite{r1985}
\[\Ru(Y,\alpha) \in \R\]
Roughly speaking, the Ruelle invariant is the integral over $Y$ of a time-averaged rotation number that measures the degree to which different Reeb trajectories twist counter-clockwise around each other (see \S \ref{subsec:Ruelle_invariant} for a detailed review). Our result is stated most elegantly using the quantity
\[\ru(Y,\alpha) = \frac{\Ru(Y,\alpha)}{\vol{Y,\alpha}^{1/2}} \]
This \emph{Ruelle ratio} is invariant under scaling of the contact form, unlike the Ruelle invariant itself.

\vspace{3pt}

In recent work \cite{h2019} motivated by embedded contact homology, Hutchings investigated the Ruelle invariant of toric domains in $\C^2$. In that paper, the Ruelle invariant of the standard ellipsoid $E = E(a,b) \subset \C^2$ with symplectic radii $0 < a \le b$ (see \S \ref{subsec:standard_ellipsoids}) was computed as
\begin{equation} \label{eqn:intro_Ru}
\Ru(E) = a + b\end{equation}
The systolic ratio $\on{sys}(E)$ and contact volume $\on{vol}(\partial E,\lambda|_{\partial E})$ are well-known to be $a/b$ and $ab$ respectively. Thus we have the following relation between the systolic and Ruelle ratios.
\[
\ru(E) = \sys(E)^{1/2} + \sys(E)^{-1/2} \quad\text{and thus}\quad 1 < \ru(E) \cdot \sys(E)^{1/2} = \sys(E) + 1 \le 2
\] 
Our first result may be viewed as a generalization of the estimate on the right to arbitrary convex contact forms on $S^3$. 

\begin{prop}[Prop \ref{prop:Ru_bound}] \label{prop:main_prop_1} There are constants $C > c > 0$ such that, for any convex contact form $\alpha$ on $S^3$, the following inequality holds.
\begin{equation} \label{eqn:Ru_lower_bound_intro} c \le \ru(S^3,\alpha) \cdot \sys(S^3,\alpha)^{1/2} \le C\end{equation}
\end{prop}

\noindent Note that a result of Viterbo \cite[Thm 5.1]{v2000} states that there exists a constant $\gamma_2$ such that $\sys(S^3,\alpha) \le \gamma_2$ for any convex contact form. Thus, Proposition \ref{prop:main_prop_1} also implies that

\begin{cor} There is a constant $c > 0$ such that, for any convex contact form $\alpha$ on $S^3$, we have
\begin{equation} \label{eqn:Ru_lower_bound_intro} c \le \ru(S^3,\alpha) \end{equation}
\end{cor}

\noindent It is notable that, even for ellipsoids, the systolic ratio can be arbitrarily close to $0$ and the Ruelle ratio can be arbitrarily close to $\infty$. We have included a helpful visualization of Proposition \ref{prop:main_prop_1} in the $\sys-\ru$ plane in Figure \ref{fig:sys_ru_plot}.

\vspace{3pt}

\begin{figure}[h]
\centering
\caption{A plot of the region of the $\sys-\ru$ plane containing convex contact forms, depicted in light red. The blue arc is the region occupied by ellipsoids, and the green lines represent the $\sys = 1$ bound and the $\sys = \gamma_2$ bound. The Viterbo conjecture states that the region of convex domains with systolic ratio larger than $1$ is empty, and so it is partially shaded in this figure.}
\vspace{5pt}
\includegraphics[width=.8\textwidth]{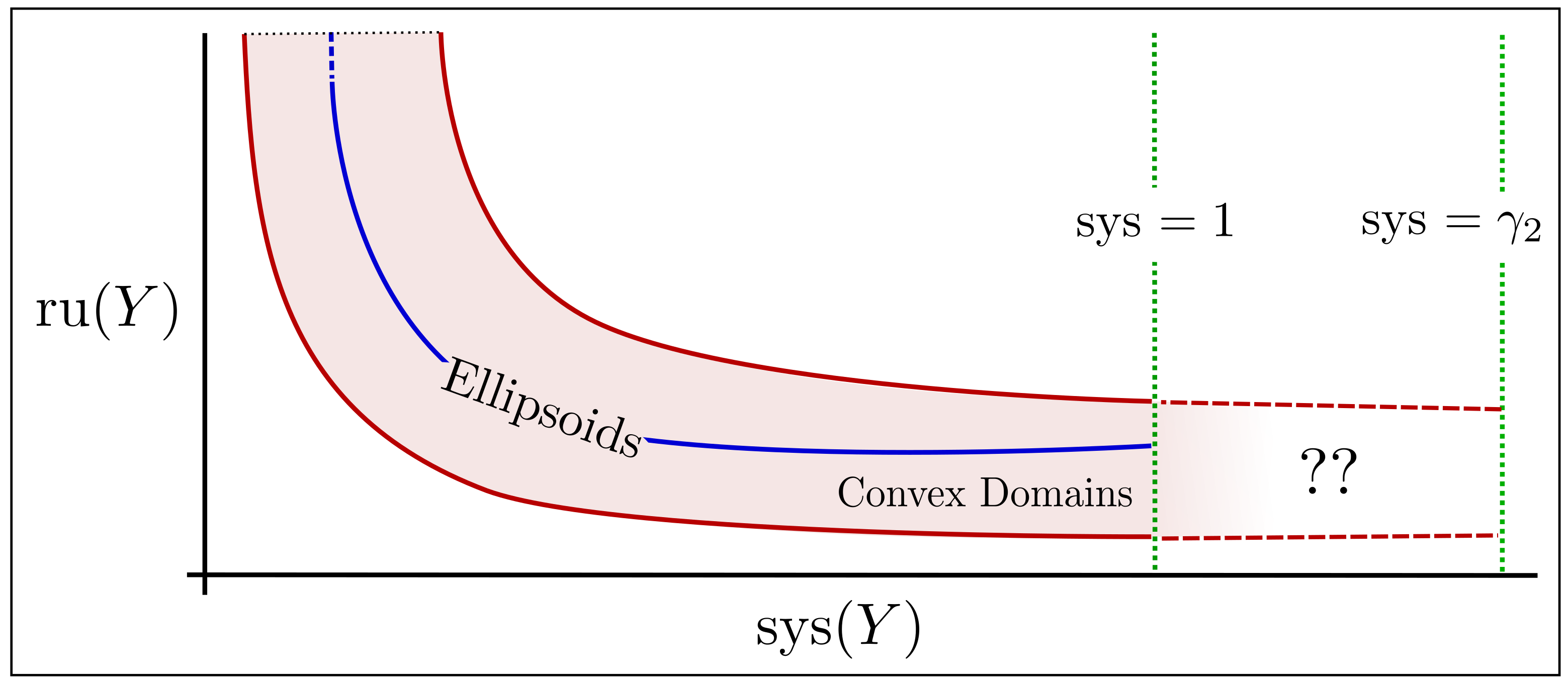}
\label{fig:sys_ru_plot}
\end{figure}

Let us explain the idea of the proof of Proposition \ref{prop:main_prop_1}. First, as explained above, the result holds for ellipsoids. By applying John's ellipsoid theorem \cite{j1948}, we can sandwich a given convex domain $X$ between an ellipsoid $E$ and its scaling $4 \cdot E$. After applying an affine symplectomorphism to $X$ and $E$, we may assume that $E$ is standard. That is
\[E(a,b) \subset X \subset 4 \cdot E(a,b)\]
Note that this symplectomorphism does not change the Ruelle invariant (see \S \ref{subsec:standard_ellipsoids}). Now note that the minimum length of a closed orbit is monotonic under inclusion of convex domains, since it coincides the the Ekeland-Hofer-Zehnder capacity in the convex setting (cf. \cite{chlrs2005}). Applying this and the monotonicity of volume, we find that
\begin{equation} \label{eqn:intro_vol_sys} \frac{ab}{2} \le \vol{X} \le 2^8 \cdot \frac{ab}{2} \qquad\text{and}\qquad 2^{-8} \cdot \frac{a}{b} \le \sys(Y) \le 2^8 \cdot \frac{a}{b}\end{equation}
If the Ruelle invariant were also monotonic, then one could immediately acquire Proposition \ref{prop:main_prop_1} from (\ref{eqn:intro_vol_sys}) and (\ref{eqn:intro_Ru}). Unfortunately, this is not evidently the case. 

\vspace{3pt}

The resolution of this issue comes from a beautiful formula (Proposition \ref{prop:rotation_curvature_formula}) relating the second fundmantal form and local rotation of the Reeb flow on a contact hypersurface $Y$ in $\R^4$. This is due originally to Ragazzo-Salom\~{a}o \cite{rs2006}, albeit in different language from this paper. Using this relation (\S \ref{subsec:curvature_rotation}), we derive estimates for the Ruelle invariant in terms of diameter, area and total mean curvature. By standard convexity theory (i.e. the theory of mixed volumes), these quantities are monotonic under inclusion of convex domains. This allows us to compare the Ruelle invariant of $X$ to that of its sandwiching ellipsoids, and thus prove the result.

\begin{remark}[Enhancing Prop \ref{prop:main_prop_1}] \label{rem:improving_main_prop_1} In future work, we plan to investigate optimal constants $c$ and $C$ for Proposition \ref{prop:main_prop_1}, and to generalize the result to higher dimensions. \end{remark}

\subsection{A Counterexample} In order to prove Theorem \ref{thm:dynamical_convex_is_not_convex} using Proposition \ref{prop:main_prop_1}, we construct dynamically convex contact forms that violate both sides of the estimate (\ref{eqn:Ru_lower_bound_intro}). This is the subject of our second new result. 

\begin{prop}[Prop \ref{prop:dyn_cvx_counterexample}] \label{prop:main_prop_2} For every $\epsilon > 0$, there exists a dynamically convex contact form $\alpha$ on $S^3$ satisfying

\[\vol{S^3,\alpha} = 1 \qquad \sys(S^3,\alpha) \ge 1 - \epsilon \qquad \Ru(S^3,\alpha) \le \epsilon\]
and there exists a dynamically convex contact form $\beta$ on $S^3$ satisfying
\[\vol{S^3,\beta} = 1 \qquad \sys(S^3,\beta) \ge 1 - \epsilon \qquad \Ru(S^3,\beta) \ge \epsilon^{-1}\]
\end{prop}

The construction of these examples follows the open book methods of Abbondandolo-Bramham-Hryniewicz-Salom\~{a}o in \cite{abhs1,abhs2}. Namely, we develop a detailed correspondence between the properties of a Hamiltonian disk map $\phi:\D \to \D$ and the properties of a contact form $\alpha$ on $S^3$ constructed using $\phi$ via the open book construction (see Proposition \ref{prop:open_book}). This includes a new formula relating the Ruelle invariant of $\phi$ in the sense of \cite{r1985} and the Ruelle invariant of $(S^3,\alpha)$. We then construct Hamiltonian disk maps $\phi$ with all of the appropriate properties to produce dynamically convex contact forms on $S^3$ satisfying the conditions in Proposition \ref{prop:main_prop_2}.

\vspace{3pt}

Let us briefly outline the construction in the small Ruelle case, as the large Ruelle case is similar. The special Hamiltonian map $\phi$ is acquired by composing two maps $\phi^H$ and $\phi^G$. The map $\phi^H$ is a counter-clockwise rotation by angle $2\pi(1 + 1/n)$ for large $n$. The map $\phi^G$ is compactly supported on a disjoint union $U$ of disks $D$, and rotates (most of) each disk $D$ clockwise about its center by angle slightly less than $4\pi$. See Figure \ref{fig:special_hamiltonian} for an illustration of this map.

\begin{figure}[h]
\centering
\caption{The map $\phi = \phi^G \circ \phi^H$ for $n = 4$. Here $\phi^H$ rotates $\D$ counter-clockwise by 90 degrees and $\phi^G$ twists each disk $D$ by roughly 720 degrees clockwise.}
\vspace{10pt}
\includegraphics[width=.8\textwidth]{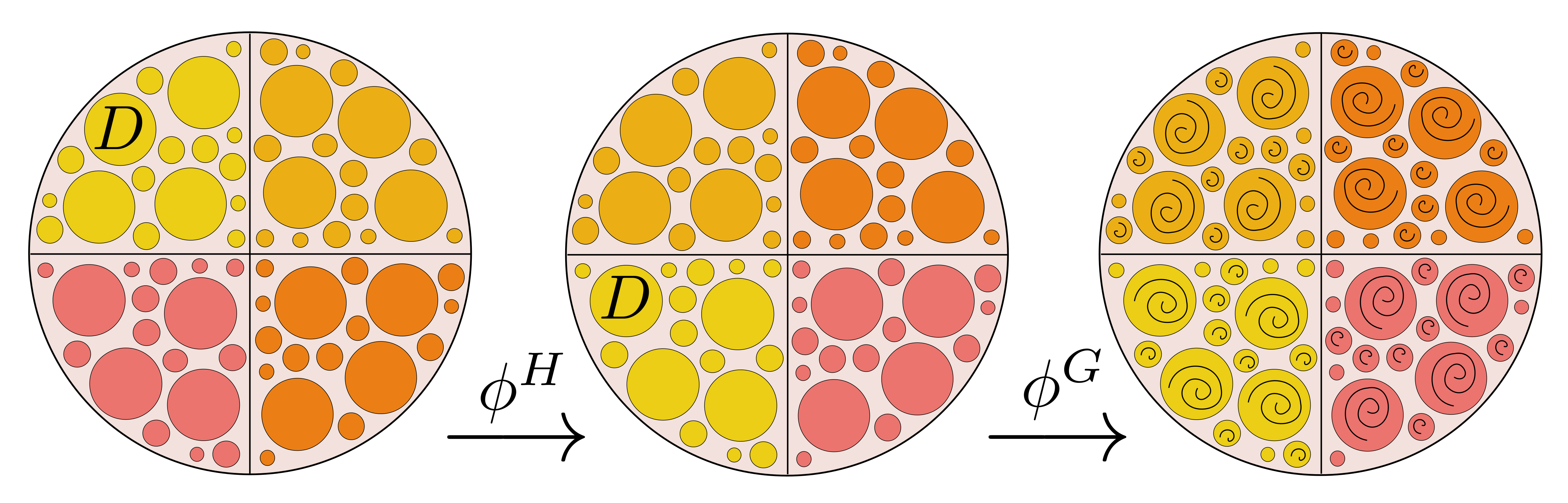}
\label{fig:special_hamiltonian}
\end{figure}

Applying Proposition \ref{prop:open_book}, we can show that the volume and Ruelle invariant of $(S^3,\alpha)$ are (up to negligible error) proportional to the following quantities.
\[\vol{S^3,\alpha} \sim \pi^2 - 2\sum_D \on{area}(D)^2 \qquad \Ru(S^3,\alpha) \sim 2\pi - 2\sum_D \on{area}(D)\]
By choosing $U$ to fill most of $\D$ and choosing all of the disks in $U$ to be very small, we can make the Ruelle invariant very small relative to the volume. This process preserves the minimal action of a closed orbit (up to a small error) and dynamical convexity, producing the desired small Ruelle invariant example.  

\begin{remark} Our examples \emph{do not} coincide with the ABHS examples in \cite{abhs2}. However, we believe that improvements of Proposition \ref{prop:main_prop_2} may make our analysis applicable to those examples.
\end{remark}

\begin{remark} In general, it is possible for the Ruelle invariant of a Reeb flow on $S^3$ to be negative. However, Proposition \ref{prop:main_prop_1} implies (via the lower bound) that the Ruelle invariant of a convex contact form is always positive. In fact, this is a much simpler property to prove than Proposition \ref{prop:main_prop_1} itself, using similar methods. However, we were not able to push the construction in \S \ref{sec:counter_example} to yield a dynamically convex contact form with non-positive Ruelle invariant.
\end{remark}

\subsection*{Outline} This concludes the introduction {\bf \S \ref{sec:introduction}}. The rest of the paper is organized as follows.

\vspace{4pt}

In {\bf \S \ref{sec:rotation_and_ruelle}}, we cover basic preliminaries needed in later sections: the rotation number (\S \ref{subsec:rotation_number}), the Conley-Zehnder index (\S \ref{subsec:CZ_index}), invariants of Reeb orbits (\S \ref{subsec:invariants_of_Reeb_orbits}) and the Ruelle invariant (\S \ref{subsec:Ruelle_invariant}). 

\vspace{4pt}

In {\bf \S \ref{sec:main_estimate}}, we prove Proposition \ref{prop:main_prop_1}. We start by discussing the curvature-rotation formula and some consequences (\S \ref{subsec:curvature_rotation}). We then derive a lower bound for a relevant curvature integral (\S \ref{subsec:lower_bounding_curvature_integral}). We conclude by proving the main bound (\S \ref{subsec:proof_of_main_bound}). 

\vspace{4pt} 

In {\bf \S \ref{sec:counter_example}}, we prove Proposition \ref{prop:main_prop_2}. We first discuss general preliminaries on Hamiltonian disk maps (\S \ref{subsec:hamiltonian_disk_maps}), open books (\S \ref{subsec:open_books}) and radial Hamiltonians (\S \ref{subsec:radial_Hamiltonians}). We then construct a Hamiltonian flow on the disk (\S \ref{subsec:special_hamiltonian_map}) before concluding with the main proof (\S \ref{subsec:main_construction}).

\subsection*{Acknowledgements} We are deeply indepted to Alberto Abbondandolo, Umberto Hryniewicz and Michael Hutchings, who explained a number of the ideas and arguments in \S \ref{sec:rotation_and_ruelle} to JC in various discussions and private communications. We would also like thank the anonymous referees for suggesting that we include counter-examples in \S \ref{sec:counter_example} that violate the upper bound in Proposition \ref{prop:Ru_bound}. JC was supported by the NSF Graduate Research Fellowship under Grant No.~1752814.

\section{Rotation Numbers And Ruelle Invariant} \label{sec:rotation_and_ruelle} In this section, we review some preliminaries on rotation numbers, Conley-Zehnder indices and the Ruelle invariant, which we will need in later parts of the paper.

\begin{remark} The rotation number, also known as the homogeneous Maslov quasimorphism, and the Conley-Zehnder index were originally introduced by Gelfand-Lidskii \cite{gl1955} albeit using different terminology. For a more contemporary perspective, see Salamon-Zehnder \cite{sz1992}. \end{remark}

\subsection{Rotation Number} \label{subsec:rotation_number} Consider the universal cover $\twSp$ of the symplectic group $\Sp$. We will view a group element $\Phi$ as a homotopy class of paths with fixed endpoints
\[\Phi:[0,1] \to \Sp \quad\text{with}\quad \Phi(0) = \on{Id}\]
Recall that a \emph{quasimorphism} $q:G \to \R$ from a group $G$ to the real line is a map such that there exists a $C > 0$ such that
\begin{equation} \label{eqn:quasimorphism_property}
|q(gh) - q(g) - q(h)| < C \qquad\text{for all }g,h \in G
\end{equation}
A quasimorphism is \emph{homogeneous} if $q(g^k) = k \cdot q(g)$ for any $g \in G$. Finally, two quasimorphisms $q$ and $q'$ are called \emph{equivalent} if the function $|q - q'|$ on $G$ is bounded. Note that any quasimorphism is equivalent to a unique homogeneous one. 

\vspace{3pt}

The universal cover of the symplectic group possesses a canonical homogeneous quasimorphism, due to the following result of Salamon-Ben Simon \cite{ss2010}.

\begin{thm}[\cite{ss2010}, Thm 1] \label{thm:quasi_homo_twSp} There exists a unique homogeneous quasimorphism
\[\rho:\twSp \to \R\]
that restricts to the following homomorphism $\rho:\twU \to \R$ on the universal cover of $\on{U}(1)$.
\begin{equation} \label{eqn:quasi_homo_twSp} \rho(\gamma) = L \qquad \text{on the path }\gamma:[0,1] \to \on{U}(1) \text{ with }\gamma(t) = \exp(2\pi i L t)\end{equation}
\end{thm} 

\begin{definition} \label{def:rotation_number} The \emph{rotation number} $\rho:\twSp \to \R$ is the quasimorphism in Theorem \ref{thm:quasi_homo_twSp}.\end{definition}

The rotation number is often characterized more explicitly in the literature as a lift of a map to the circle. More precisely, it is characterized as the unique lift 
\begin{equation}
\widetilde{\sigma}:\twSp \to \R \qquad\text{of}\qquad \sigma:\Sp \to S^1 \qquad \text{such that}\qquad \widetilde{\sigma}(\on{Id}) = 0\end{equation}
via the covering map $\R\rightarrow S^1\subset\C$ given by $\theta\mapsto e^{2\pi i \theta}$. Here $\sigma$ is defined as follows. Let $\Phi \in \Sp$ have real eigenvalues $\lambda,\lambda^{-1}$ and let $\Psi \in \Sp$ have complex (unit) eigenvalues $\zeta,\overline{\zeta}$ with $\operatorname{Im}\zeta >0$. Also fix an arbitrary $v \in \R^2 \setminus 0$, identified with an element of $\C$ in the usual way. Then
\begin{equation}\label{eqn:rotation_number_mod_1}
\sigma(\Phi) = 
\left\{\begin{array}{cc}
\vspace{4pt}
0 & \text{ if } \lambda > 0\\
1/2 & \text{ if } \lambda < 0\\
\end{array}\right. \quad\text{and}\quad
\sigma(\Psi) = 
\left\{\begin{array}{cc}
\vspace{4pt}
\zeta & \text{ if } \langle iv,\Phi v\rangle > 0\\
\overline{\zeta} & \text{ if } \langle iv,\Phi v\rangle < 0\\
\end{array}\right. 
\end{equation}
Here $iv$ denotes multiplication of $v$ by $i \in \C$, i.e. the rotation of $v$ by 90 degrees counterclockwise. All of the elements of $\Sp$ fall into one of the two categories above, and so $\sigma$ is determined everywhere by (\ref{eqn:rotation_number_mod_1}).

\begin{lemma} \label{lem:rotation_lift_formula} The rotation number $\rho:\twSp \to \R$ is the lift of $\sigma:\Sp \to \R/\Z$ with $\rho(\operatorname{Id}) = 0$.  
\end{lemma}

\begin{proof} We verify the properties in Theorem \ref{thm:quasi_homo_twSp}. The lift $\widetilde{\sigma}$ is a quasimorphism by Lemmas \ref{lem:rhos_all_equivalent} and \ref{lem:rho_s_is_quasimorphism} below. Also note that since eigenvalues and the sign of $\langle iv,\Phi v\rangle$ are invariant under conjugation, $\sigma$ is as well.

\vspace{3pt}

To check that $\widetilde{\sigma}$ is homogeneous, note that if $\Phi$ has real eigenvalues of sign $s = \pm 1$, then $\sigma(\Phi^k) = \frac{1}{2}(1 - s^k) = ks \mod 1$. On the otherhand, if $\Phi$ has complex unit eigenvalues $\zeta,\bar{\zeta}$ for $\on{Re}(\zeta) > 0$, then it is conjugate to a rotation $\exp(2\pi i \theta) \in U(1) \subset \Sp$ on $\C \simeq \R^2$ and thus
\[\sigma(\Phi^k) = \sigma(\exp(2\pi i k\theta)) = k\theta \mod 1\]
Thus $\sigma(\Phi^k) = k\cdot\sigma(\Phi) \mod 1$ and the lift satisifes $\widetilde{\sigma}(\Phi^k) = k \widetilde{\sigma}(\Phi)$. Finally, if $\gamma:[0,1] \to \Sp$ is given by $\gamma(t) = \exp(2\pi i Lt)$ then
\[\sigma \circ \gamma:[0,1] \to \R/\Z \quad\text{is given by} \quad \sigma \circ \gamma(t) = Lt \mod 1 \in \R/\Z\]
This implies that the lift is $t \mapsto Lt$, so that $\widetilde{\sigma}(\gamma) = L$. This proves the needed criteria.  
\end{proof}

\vspace{3pt}

We will also need to utilize several inhomogeneous versions of the rotation number depending on a choice of unit vector. These are defined a follows.

\begin{definition} \label{def:rotation_number_rel_s} The \emph{rotation number} $\rho_s:\twSp \to \R$ \emph{relative to $s \in S^1$} is the unique lift of the map
\[\sigma_s:\Sp \to S^1 \qquad \Phi \mapsto |\Phi s|^{-1} \cdot \Phi s \in S^1 \subset \R^2\] via the covering map $\R \to S^1 \subset \C$ given by $\theta \mapsto e^{2\pi i \theta}\cdot s$ such that $\rho_s(\on{Id})=0$. Here $\Phi s$ denotes the application of the matrix $\Phi \in \Sp$ to the unit vector $s \in S^1$.
\end{definition}

The rotation numbers relative to $s \in S^1$ and the lift of $\sigma$ all have bounded difference from one another. Precisely, we have the following lemma.

\begin{lemma} \label{lem:rhos_all_equivalent} The maps $\rho_s:\twSp \to \R$ and the lift $\widetilde{\sigma}:\twSp \to \R$ of $\sigma$ have bounded difference. More precisely, we have the following bounds.
\begin{equation} \label{eqn:rhos_all_equivalent}
|\rho_s - \widetilde{\sigma}| \le 1 \qquad\text{and}\qquad |\rho_s - \rho_t| \le 1 \qquad \text{for any pair }s,t \in S^1
\end{equation}
\end{lemma}

\begin{proof} First, assume that $\Phi:[0,1] \to \Sp$ is a path such that $\Phi(t)$ has no negative real eigenvalues for any $t \in [0,1]$. Then
\[\sigma \circ \Phi(t) \neq 1/2 \qquad\text{and}\qquad \sigma_s \circ \Phi(t) \neq -s \in S^1 \qquad\text{ for any }s \in S^1 \text{ and }t \in [0,1]\]
It follows that the relevant lifts of $\sigma \circ \Phi$ and $\sigma_s \circ \Phi$ to maps $[0,1] \to \R$ remain in the interval $(-1/2,1/2)$ for all $t$. Thus
\[\widetilde{\sigma}(\Phi) \in (-1/2,1/2) \qquad\text{and}\qquad \rho_s(\Phi) \in (-1/2,1/2)\]
This clearly implies (\ref{eqn:rhos_all_equivalent}) since $\Phi(t)$ does not have any negative eigenvalues. Since $\sigma$ induces an isomorphism $\pi_1(\Sp) \to \pi_1(S^1)$, we know that for any pair $\Phi,\Phi' \in \twSp$ with the same projection to $\Sp$
\[\widetilde{\sigma}(\Phi) = \widetilde{\sigma}(\Phi') \qquad\text{implies}\qquad \Phi = \Phi'\]
In particular, the above analysis extends to any $\Phi$ with $\widetilde{\sigma}(\Phi) \in (-1/2,1/2)$. In the general case, note that the path $\gamma:[0,1] \to S^1$ given by $\gamma(t) = \exp(\pi i \cdot k t)$ for an integer $k \in \Z$ satisfies
\[\widetilde{\sigma}(\gamma) = \rho_s(\gamma) = k/2 \qquad \widetilde{\sigma}(\Phi\gamma) = \widetilde{\sigma}(\Phi) + \widetilde{\sigma}(\gamma) \qquad \rho_s(\Phi \gamma) = \rho_s(\Phi) + \rho_s(\gamma)\]
Any path $\Psi$ can be decomposed (up to homotopy) as $\Phi \gamma$ where $\gamma$ is as above and $\Phi:[0,1] \to \Sp$ is a path with $\widetilde{\sigma}(\Phi) \in (-1/2,1/2)$. This reduces to the special case. 
\end{proof}

This can be used to demonstrate that $\rho_s$ is a quasimorphism. As noted in the proof of Lemma \ref{lem:rotation_lift_formula}, this implies that $\widetilde{\sigma}$ is a quasimorphism as well.

\begin{lemma} \label{lem:rho_s_is_quasimorphism} The map $\rho_s:\twSp \to \R$ is a quasimorphism for any $s \in S^1$. In fact, we have
\begin{equation} \label{eqn:rho_s_all_equivalent}
|\rho_s(\Psi\Phi) - \rho_s(\Psi) - \rho_s(\Phi)| \le 1 \qquad \text{for any}\qquad s\in S^1\end{equation}
\end{lemma}

\begin{proof} Let $\Phi:[0,1] \to \Sp$ and $\Psi:[0,1] \to \Sp$ be two elements of $\twSp$ viewed as paths in $\Sp$. Consider the product $\Psi\Phi$ in the universal cover of $\Sp$, represented by the path
\[\Phi(2t)\text{ for }t \in [0,1/2] \qquad\text{and}\qquad \Psi(2t-1)\Phi(1) \text{ for }t \in [1/2,1] \]
By examining the path $\sigma_s \circ \Psi\Phi:[0,1] \to S^1$ and the lift to $\R$, we deduce the following property.
\begin{equation}\label{eqn:rho_s_additivity_property} \rho_s(\Psi\Phi) = \rho_{\Phi(s)}(\Psi) + \rho_s(\Phi)\end{equation}
Here $\Phi(s)$ is shorthand for the unit vector $\Phi(1)s/|\Phi(1)s|$. Applying Lemma \ref{lem:rhos_all_equivalent}, we have
\[|\rho_s(\Psi\Phi) - \rho_s(\Psi) - \rho_s(\Phi)| \le |\rho_{\Phi(s)}(\Psi) - \rho_s(\Psi)| \le 1\]
This proves the quasimorphism property.\end{proof}

\subsection{Conley-Zehnder Index} \label{subsec:CZ_index} Let $\Spx \subset \Sp$ denote the subset of elements $\Phi \in \Sp$ such that $\Phi - \on{Id}$ is invertible, and let $\twSpx$ be the inverse image of $\Spx$ under $\pi:\twSp \to \Sp$. 

\vspace{3pt}

The \emph{Conley-Zehnder index} is a natural integer invariant of paths in $\Spx$, denoted as follows.
\[
\CZ:\twSpx \to \Z
\]
This invariant was introduce in \cite{gl1955} (also see \cite{sz1992}). We will use the following formula as our definition throughout this paper.
\begin{equation}\label{eqn:CZ_w_rotation_number_on_Spx}
\CZ(\Phi) = \lfloor \rho(\Phi) \rfloor + \lceil \rho(\Phi) \rceil
\end{equation}
There are several inequivalent ways to extend the Conley-Zehnder index to the entire symplectic group. We will follow \cite[\S 3]{hwz1998} and \cite[\S 2.2]{abhs2}, and use the following extension. 

\begin{convention} \label{con:CZ_on_Sp} In this paper, the \emph{Conley-Zehnder index} $\CZ:\twSp \to \Z$ will be the maximal lower semi-continuous extension of the ordinary Conley-Zehnder index.
\end{convention}

\noindent The extension in Convention \ref{con:CZ_on_Sp} can be bounded from below in terms of the rotation number.

\begin{lemma} \label{lem:CZ_index_3_vs_rho} Let $\Phi \in \twSp$. Then
\begin{equation} \label{eqn:CZ_w_rotation_number_on_Sp}
\CZ(\Phi) \ge 2 \cdot \lceil \rho(\Phi) \rceil - 1
\end{equation}
\end{lemma}

\begin{proof} For $\Phi \in \twSpx$, (\ref{eqn:CZ_w_rotation_number_on_Sp}) is an immediate consequence of (\ref{eqn:CZ_w_rotation_number_on_Spx}). In the other case, note that the maximal lower semi-continuous extension is defined by the property that
\[
\CZ(\Phi) = \liminf_{\Psi\rightarrow\Phi} \CZ(\Psi) \qquad\text{for any }\Phi \not\in\twSpx
\]
Any $\Phi \not\in \twSpx$ has eigenvalue $1$, and so Lemma \ref{lem:rotation_lift_formula} implies that $\rho(\Phi) \in \Z$. Since $\rho$ is continuous, we find that
\[
\CZ(\Phi) = \liminf_{\Psi\rightarrow\Phi} \lfloor \rho(\Psi) \rfloor + \lceil \rho(\Psi) \rceil \ge  \lfloor \rho(\Phi) - 1/2 \rfloor + \lceil \rho(\Phi) - 1/2 \rceil = 2 \cdot \lceil \rho(\Phi) \rceil - 1
\]
This proves the lower bound in every case. \end{proof}

\subsection{Invariants Of Reeb Orbits} \label{subsec:invariants_of_Reeb_orbits} Let $(Y,\xi)$ be a closed contact $3$-manifold with $c_1(\xi) = 0$ and let $\alpha$ be a contact 1-form on $Y$. 

\vspace{3pt}

Under this hypothesis on the Chern class, $\xi$ is isomorphic as a symplectic vector-bundle to the trivial bundle $\R^2$. A \emph{trivialization} $\tau$ of $\xi$ is a bundle isomorphism
\[\tau:\xi \simeq \R^2 \qquad\text{denoted by}\qquad \tau(y):\xi_y \simeq \R^2 \qquad\text{satisfying}\qquad \tau(y)^*\omega = d\alpha|_\xi\]
Two trivializations are \emph{homotopic} if they are connected by a 1-parameter family of bundle isomorphisms. Given a trivialization $\tau$, we may associate a \emph{linearized} Reeb flow
\begin{equation} \label{eqn:linear_Reeb_flow_wrt_tau} \Phi_\tau:\R \times Y \to \Sp \qquad\text{given by}\qquad \Phi_\tau(T,y) = \tau(\phi_T(y)) \circ d\phi_T(y) \circ \tau^{-1}(y)\end{equation}
Here $\phi:\R \times Y \to Y$ is the Reeb flow, i.e. the flow generated by the Reeb vector field $R$, and we use the notation $\phi_T(y)=\phi(T,y)$. The linearized flow lifts uniquely to a map
\[
\widetilde{\Phi}_\tau:\R \times Y \to \twSp \qquad \text{ with } \qquad \widetilde{\Phi}_\tau|_{0 \times Y} = \on{Id} \in \twSp
\]
We will refer to $\widetilde{\Phi}_\tau$ as the \emph{lifted} linearized Reeb flow. Explicitly, it maps $(T,y)$ to the homotopy class of the path $\Phi_\tau(\cdot,y)|_{[0,T]}$. Note that this lift satisfies the cocyle property
\begin{equation} \label{eqn:linearized_flow_cocycle}
\widetilde{\Phi}_\tau(S+T,y) = \widetilde{\Phi}_\tau(T,\phi_S(y)) \cdot \widetilde{\Phi}_\tau(S,y)
\end{equation}

\vspace{3pt}

\begin{definition} Let $\gamma:\R/L\Z \to Y$ be a closed Reeb orbit of $Y$. The \emph{action} of $\gamma$ is given by
\begin{equation}
\label{eqn:action_of_orbit} 
\mathcal{A}(\gamma) = \int \gamma^*\alpha = L\end{equation}
Likewise, the \emph{rotation number} and \emph{Conley-Zehnder index} of $\gamma$ with respect to $\tau$ are given by 
\begin{equation}
\rho(\gamma,\tau) := \rho \circ \widetilde{\Phi}_\tau(L,y) \qquad \CZ(\gamma,\tau) := \CZ(\widetilde{\Phi}_\tau(L,y)) \qquad\text{where }y = \gamma(0)
\end{equation}
These invariants depend only on the homotopy class of $\tau$, and if $H^1(Y;\Z) = 0$ (e.g. if $Y$ is the 3-sphere) there is a unique trivialization up to homotopy. In this case, we let
\begin{equation}
\label{eqn:rho_and_CZ_of_orbit} 
\rho(\gamma) := \rho(\gamma,\tau) \qquad\text{and}\qquad  \CZ(\gamma) := \CZ(\gamma,\tau) \qquad\text{for any $\tau$}
\end{equation}
\end{definition}

In \S \ref{sec:counter_example}, we will need the following easy observation, which follows immediately from Lemma \ref{lem:CZ_index_3_vs_rho} and our way of defining $\CZ$ (see Convention \ref{con:CZ_on_Sp}).

\begin{lemma} Let $\alpha$ be a contact form on $S^3$ with $\rho(\gamma) > 1$ for every closed Reeb orbit. Then $\alpha$ is dynamically convex.
\end{lemma}

\subsection{Ruelle Invariant} \label{subsec:Ruelle_invariant} Let $(Y,\xi)$ be a closed contact $3$-manifold with $c_1(\xi) = 0$ equipped with a contact form $\alpha$ and a homotopy class of trivialization $[\tau]$ of $\xi$. Here we discuss the \emph{Ruelle invariant}
\[\Ru(Y,\alpha,[\tau]) \in \R\] 
associated to the data of $Y,\alpha$ and $[\tau]$. 

\begin{remark} This invariant was originally introduced by Ruelle in \cite{r1985} for area preserving diffeomorphisms of surfaces and volume preserving flows on 3-manifolds. Variants of this construction have also appeared under different names in other settings, e.g. as the asymptotic Maslov index \cite[p. 1423]{cgip2003}. \end{remark}

It will be helpful to describe a more general construction that subsumes that of the Ruelle invariant. For this purpose, we also fix a continuous quasimorphism
\[q:\twSp \to \R\]
Pick a representative trivialization $\tau$ of $[\tau]$ and let $\widetilde{\Phi}_\tau:\R\times Y \to \twSp$ be the lifted linearized Reeb flow. We can associate a time-averaged version of $q$ over the space $Y$, as follows. 

\begin{prop} \label{prop:general_ruelle_invariant} The $1$-parameter family of functions $f_T:Y \to \R$ given by the formula
\begin{equation}
f_T(y) := \frac{q \circ \widetilde{\Phi}_\tau(T,y)}{T}
\end{equation}
converges in $L^1(Y;\R)$ and almost everywhere to a function $f(\alpha,q,\tau):Y \to \R$ with the following properties.
\begin{itemize}
	\item[(a)] (Quasimorphism) If $q$ and $r$ are equivalent quasimorphisms, i.e. $|q - r|$ is bounded, then
	\[f(\alpha,q,\tau) = f(\alpha,r,\tau)\]
	\item[(b)] (Trivialization) If $\sigma$ and $\tau$ are homotopic trivializations of $\xi$, then
	\[f(\alpha,q,\sigma) = f(\alpha,q,\tau)\]
    \item[(c)] (Contact Form) The integral $F(\alpha)$ of $f(\alpha,q,\tau)$ over $Y$ is continuous in the $C^2$-topology on $\Omega^1(Y)$.
\end{itemize}
\end{prop}

In order to prove the existence part of this result, we will need to use a verison of Kingman's subadditivity theorem appearing in \cite{k1973}. 

\begin{thm} \label{thm:kingman_ergodic} Let $Y$ be a measure space and let $\phi:\R \times Y \to Y$ be a flow with invariant measure $\mu$. Let $g_T:Y \to \R$ for $T \in \R$ be a family of $L^1(Y,\mu)$ functions such that, for some constants $C,D > 0$, we have
\[g_{S+T} \le g_S + \phi^*_Sg_T + C \qquad\qquad \int_Y g_T\cdot \mu \ge -D \cdot T \qquad\qquad \int_Y \big(\sup_{0 \le S \le 1} |g_S|\big) \cdot \mu < \infty\]
Then the maps $\frac{g_T}{T}$ converge in $L^1(Y,\mu)$ and pointwise almost everywhere as $T \to \infty$.
\end{thm}

\begin{remark} \label{rmk:discrete_kingman} There is also a version of Theorem \ref{thm:kingman_ergodic} in \cite{k1973} for a discrete dynamical system, i.e a map $\phi:Y \to Y$ preserving $\mu$. The statement is directly analogous to Theorem \ref{thm:kingman_ergodic}, but the last condition on the integrability of $\sup_{0 \le S \le 1} |g_S|$ is unnecessary. We will use this version in \S \ref{subsec:hamiltonian_disk_maps}. \end{remark}

\begin{remark} This statement is a slight variation of Theorem 4 in \cite{k1973}, which states the result for general sub-additive processes. Our version follows from the discussion in \S 1.3 of \cite{k1973} for the continuous parameter space $\R$. Note that we also weaken sub-additivity by allowing $g_T$ to be sub-additive with respect to $T$ up to an overall constant factor. \end{remark}

\begin{proof} (Proposition \ref{prop:general_ruelle_invariant}) We prove the existence of the limit and the properties (a)-(c) separately.

\vspace{3pt}

{\bf Convergence.} We apply Kingman's ergodic theorem, Theorem \ref{thm:kingman_ergodic}. Fix a constant $C > 0$ for the quasimorphism $q$ satisfying (\ref{eqn:quasimorphism_property}). Let $g_T$ denote the function on $Y$ given by
\[g_T := Tf_T = q \circ \widetilde{\Phi}_\tau(T,-)\]
Now we verify the properties in Theorem \ref{thm:kingman_ergodic}. First, due to the cocycle property (\ref{eqn:linearized_flow_cocycle}) we have
\begin{equation}\label{eqn:subadditive_process_a}
g_{S + T} = q \circ \widetilde{\Phi}_\tau(S+T,-) \le q \circ \widetilde{\Phi}_\tau(S,-) + q \circ \widetilde{\Phi}_\tau(T,\phi_S(-)) + C = g_S + \phi^*_Sg_T + C \end{equation}
We can analogously show that $g_{S+T} \ge g_S + \phi^*_Sg_T - C$. In particular, if $T > 0$ is a sufficiently large time with $T = n + S$ and $S \in [0,1]$, then 
\begin{equation}\label{eqn:subadditive_process_b}
\int_Y g_T \cdot \alpha \wedge d\alpha \ge \sum_{k=0}^{n-1}\int_Y \phi^*_kg_1  \cdot \alpha \wedge d\alpha + \int_Y \phi_n^* g_S \cdot \alpha \wedge d\alpha - CT \ge -AT\end{equation}
Here $A$ is any number larger than $C$ and larger than the quantity
\[
-\on{min}\left\{\int_Y g_S \cdot \alpha \wedge d\alpha \; : \; S \in [0,1]\right\}
\]
Finally, since $q \circ \widetilde{\Phi}_\tau$ is continuous on $\R \times Y$, it is clear that $\on{sup}_{T \in [0,1]} |f_T|$ is continuous and bounded. In particular, it is integrable. Thus $g_T$ satisfies the criteria in Theorem \ref{thm:kingman_ergodic}, and we may conclude that $\frac{g_T}{T}$ converges in $L^1$ and almost everywhere to a map
\[
f(\alpha,q,\tau) \in L^1(Y;\R)
\]

\vspace{3pt}

{\bf Quasimorphisms.} Let $q$ and $r$ be equivalent quasimorphisms, and pick $C > 0$ such that $|q - r| < C$ everywhere. Then
\[
\|\frac{q \circ \widetilde{\Phi}_\tau}{T} - \frac{r \circ \widetilde{\Phi}_\tau}{T}\|_{L^1} \le \frac{C \cdot \vol{Y,\alpha}}{T}
\]
Taking the limit as $T \to \infty$ shows that the limiting functions $f(\alpha,q,\tau)$ and $f(\alpha,r,\tau)$ are equal.

\vspace{3pt}

{\bf Trivializations.} Let $\sigma$ and $\tau$ be two trivializations of $\xi$ in the homotopy class $[\tau]$. Then there is a transition map $\Psi:Y \to \Sp$ given by
\[\Psi(y):\R^2 \to \R^2 \qquad\text{with}\qquad \Psi(y) = \tau(y) \cdot \sigma(y)^{-1}\]
The linearized flows of $\sigma$ and $\tau$ are related via this transition map, by the following formula.
\[
\Phi_\tau(T,y) = \Psi(\phi(T,y)) \cdot \Phi_\sigma(T,y) \cdot \Psi^{-1}(y)
\]
Since $\sigma$ and $\tau$ are homotopic, $\Psi$ is homotopic to a constant map. In particular, $\Psi$ lifts to the universal cover of $\Sp$. Thus we may write
\[
\widetilde{\Phi}_\tau(T,y) = \widetilde{\Psi}(\phi(T,y)) \cdot \widetilde{\Phi}_\sigma(T,y) \cdot \widetilde{\Psi}^{-1}(y)
\]
Here $\widetilde{\Psi}:Y \to \twSp$ is any lift of $\Psi$. The quasimorphism property of $\rho$ now implies that
\[
\|\frac{q \circ \widetilde{\Phi}_\sigma(T,y)}{T} - \frac{q \circ \widetilde{\Phi}_\tau(T,y)}{T}\|_{L^1} \le \frac{2C + \sup |q \circ \widetilde{\Psi}| + \sup |q \circ \widetilde{\Psi}^{-1}|}{T} \cdot \vol{Y,\alpha}
\]
Taking the limit as $T \to \infty$ shows that $f(\alpha,q,\sigma) = f(\alpha,q,\tau)$. 

\vspace{3pt}

{\bf Contact Form.} Fix a contact form $\alpha$ and an $\epsilon > 0$. Since $q$ is a quasimorphism, there exists a $C > 0$ depending only on $q$ such that
\[
|\rho \circ \widetilde{\Phi}_\tau(nT,y) - \sum_{k=0}^{n-1} \rho \circ \widetilde{\Phi}_\tau(T,\phi_T^k(y))| \le Cn \qquad \text{ for any } n,T > 0
\]
We can divide by $nT$ and rewrite this estimate in terms of $f_T$ to see that
\[
|f_{nT} - \frac{1}{n}\sum_{k=0}^{n-1} f_T \circ \phi_T^k| \le \frac{C}{T} \qquad \text{ for any } n,T > 0
\]
We can then integrate over $Y$ and take the limit as $n \to \infty$ to acquire
\begin{equation} \label{eqn:prop:general_ruelle_invariant_a}
|F(\alpha) - \int_Y f_T \cdot \alpha \wedge d\alpha| = \lim_{n \to \infty} |\int_Y (f_{nT} - f_T) \cdot \alpha \wedge d\alpha|\end{equation}
\[ = \lim_{n \to \infty} |\int_Y (f_{nT} - \frac{1}{n}\sum_{k=0}^{n-1} f_T \circ \phi_T^k) \cdot \alpha \wedge d\alpha| \le \frac{C \cdot \vol{Y,\alpha}}{T}
\]
We use the fact that $\phi_T$ preserves $\alpha \wedge d\alpha$ in moving from the first to the second line above.

Next, fix a different contact form $\beta$. Let $\widetilde{\Psi}_\tau$ be the lifted linearized flow for $\beta$, and let 
\[g_T:Y \to \R \quad\text{where}\quad g_T(y) = \frac{q \circ \widetilde{\Psi}_\tau(T,-)}{T}\]
Due to (\ref{eqn:prop:general_ruelle_invariant_a}), we can fix a $T > 0$ such that, for all $\beta$ sufficiently $C^0$-close to $\alpha$, we have
\begin{equation} \label{eqn:prop:general_ruelle_invariant_b} 
|F(\alpha) - \int_Y f_T  \cdot \alpha \wedge d\alpha| < \frac{\epsilon}{3} \quad\text{and}\quad
|F(\beta) - \int_Y g_T  \cdot \beta \wedge d\beta| < \frac{2C\vol{Y,\alpha}}{T} < \frac{\epsilon}{3}\end{equation}
Furthermore, we may bound the integrals of $f_T$ and $g_T$ as follows.
\begin{equation} \label{eqn:prop:general_ruelle_invariant_c}
|\int_Y f_T  \cdot \alpha \wedge d\alpha - \int_Y g_T  \cdot \beta \wedge d\beta| \le \int_Y |f_T - g_T|  \cdot \alpha \wedge d\alpha + |\int_Y f_T  \cdot (\alpha \wedge d\alpha - \beta \wedge d\beta)|
\end{equation}
\[
\le  \|f_T - g_T\|_{C^0(Y)} \cdot \vol{Y,\alpha} + \|g_T\|_{C^0(Y)} \cdot |\vol{Y,\alpha} - \vol{Y,\beta}|
\]
We can choose $\beta$ sufficiently close to $\alpha$ in $C^2(Y)$ so that $\widetilde{\Psi}_\tau$ is arbitrarily $C^0$-close to $\widetilde{\Phi}_\tau$ on $[0,T]\times Y$. Since $Y$ is compact, the image of $\widetilde{\Phi}(T,-)$ is compact in $\twSp$ for fixed time $T$. Thus, since $q$ is continuous, $g_T = q \circ \widetilde{\Psi}_\tau(T,-)$ can also be made arbitrarily $C^0$-close to $f_T = q \circ \widetilde{\Phi}_\tau(T,-)$. In particular, for $\beta$ sufficiently $C^2$-close to $\alpha$ we have
\begin{equation} \label{eqn:prop:general_ruelle_invariant_d}
\|f_T - g_T\|_{C^0(Y)} \cdot \vol{Y,\alpha} + \|g_T\|_{C^0(Y)} \cdot |\vol{Y,\alpha} - \vol{Y,\beta}| < \frac{\epsilon}{3}
\end{equation}
Together, (\ref{eqn:prop:general_ruelle_invariant_b}), (\ref{eqn:prop:general_ruelle_invariant_c}) and (\ref{eqn:prop:general_ruelle_invariant_d}) imply that, for $\beta$ sufficiently $C^2$-close to $\alpha$, we have $|F(\alpha) - F(\beta)| < \epsilon$. This proves continuity.

\vspace{4pt}

\noindent This concludes the proof of the existence and properties of $f(\alpha,q,\tau)$, and of Proposition \ref{prop:general_ruelle_invariant}.\end{proof}

Proposition \ref{prop:general_ruelle_invariant} allows us to introduce the Ruelle invariant as an integral quantity, as follows.

\begin{definition}[Ruelle Invariant] The \emph{local rotation number} $\rot_\tau$ of a closed contact manifold $(Y,\alpha)$ equipped with a (homotopy class of) trivialization $\tau$ is the following limit in $L^1$.
\begin{equation} \label{eqn:general_ruelle_invariant} \rot_\tau:Y \to \R \qquad\text{given by}\qquad \rot_\tau := \lim_{T \to \infty} \frac{\rho \circ \widetilde{\Phi}_\tau(T,-)}{T}\end{equation}
Similarly, the \emph{Ruelle invariant} $\Ru(Y,\alpha,\tau)$ is the integral of the local rotation number over $Y$, i.e.
\begin{equation} \label{eqn:Ruelle_invariant}
\Ru(Y,\alpha,\tau) = \int_Y \rot_\tau \cdot \alpha \wedge d\alpha = \lim_{T \to \infty} \frac{1}{T} \int_Y \rho \circ \widetilde{\Phi}_\tau \cdot \alpha \wedge d\alpha
\end{equation}
\end{definition} 

We will require an alternative expression for the Ruelle invariant in order to derive estimates later in the paper. 

\vspace{3pt}

The Reeb flow $\phi$ on $Y$ preserves the contact structure, and so lifts to a flow on the total space of the contact structure $\xi$. Since this flow is fiberwise linear, it descends to the (oriented) projectivization $P\xi$. A trivialization $\tau$ determines an identification $P\xi \simeq Y \times \R/\Z$, and so a flow
\begin{equation}\label{eqn:lifted_flow_Y_R_Z} \bar{\Phi}:\R \times Y \times \R/\Z \to Y \times \R/\Z \qquad \text{generated by a vector field $\bar{R}$ on $Y \times \R/\Z$}\end{equation}
Let $\theta:Y \times \R/\Z \to \R/\Z$ denote the tautological projection.

\begin{definition} \label{def:rotation_density} The \emph{rotation density} $\varrho_\tau:Y \times \R/\Z \to \R$ is the Lie derivative
\begin{equation}\varrho_\tau := \mathcal{L}_{\bar{R}}(\theta)\end{equation}
\end{definition}

\begin{lemma} \label{lem:Ru_alternate_formula} The Ruelle invariant $\Ru(Y,\alpha,\tau)$ is written using the rotation density $\varrho_\tau$ as
\[
\Ru(Y,\alpha,\tau) = \lim_{T \to \infty} \frac{1}{T}\int_0^T \big(\int_Y \bar{\Phi}_t^*\varrho_\tau(-,s) \cdot \alpha\wedge d\alpha\big) dt \qquad\text{for any fixed }s \in \R/\Z
\]
\end{lemma}

\begin{proof} By comparing Definition \ref{def:rotation_number_rel_s} with the formula (\ref{eqn:lifted_flow_Y_R_Z}), one may verify that 
\[\sigma_s \circ \Phi_\tau(T,y) \quad\text{and}\quad \theta \circ \bar{\Phi}(T,y,s) - s \quad \text{ are equal in } \R/\Z\]
Therefore, these formulas define a single map $\R \times Y \times \R/\Z \to \R/\Z$, admitting a unique lift to a map $F:\R \times Y \times \R/\Z  \to \R$ that vanishes on $0 \times Y \times \R/\Z$. The first formula implies that
\begin{equation} \label{eqn:Ru_alternate_formula_a}
F(T,y,s) = \rho_s \circ \widetilde{\Phi}_\tau(T,y)\end{equation}
On the other hand, let $t$ be the $\R$-variable of $F$ and $\theta \circ \bar{\Phi}$. Then the $t$-derivative of $F$ is 
\[\frac{dF}{dt}|_T = \frac{d}{dt}(\theta \circ \bar{\Phi})|_T = \bar{\Phi}_T^*(\mathcal{L}_{\bar{R}}(\theta)) = \bar{\Phi}_T^*\varrho_\tau \]
Integrating this identity and combining it with (\ref{eqn:Ru_alternate_formula_a}), we acquire the formula
\begin{equation} \label{eqn:Ru_alternate_formula_b}
\rho_s \circ \widetilde{\Phi}_\tau(T,y) = F(T,y,s) = \int_0^T \bar{\Phi}_t^*\varrho_\tau(y,s) \cdot dt
\end{equation} \label{eqn:Ru_alternate_formula_c}
Now, since $\rho_s$ and $\rho$ are equivalent by Lemma \ref{lem:rhos_all_equivalent}, we can apply Proposition \ref{prop:general_ruelle_invariant}(a) to see that
\begin{equation}\Ru(Y,\alpha,\tau) = \lim_{T \to \infty} \int_Y \frac{\rho_s \circ \widetilde{\Phi}_\tau(T,-)}{T} \cdot \alpha \wedge d\alpha\end{equation}
We then apply (\ref{eqn:Ru_alternate_formula_b}) and Fubini's theorem to see that the righthand side is given by
\begin{equation} \label{eqn:Ru_alternate_formula_d}
\lim_{T \to \infty} \frac{1}{T}\int_Y \big(\int_0^T \bar{\Phi}_t^*\varrho_\tau(-,s)  dt \big) \alpha \wedge d\alpha = \lim_{T \to \infty} \frac{1}{T} \int_0^T \big(\int_Y \bar{\Phi}_t^*\varrho_\tau(-,s) \cdot \alpha \wedge d\alpha\big) dt\end{equation}
This concludes the proof. \end{proof}

\section{Bounding The Ruelle Invariant} \label{sec:main_estimate} Let $X \subset \R^4$ be a convex, star-shaped domain with smooth contact boundary $(Y,\lambda)$. In this section, we derive the following estimate for the Ruelle ratio. 

\begin{prop} \label{prop:Ru_bound} There exist positive constants $c$ and $C$ independent of $Y$ such that
\begin{equation*}
c \le \ru(Y,\lambda)\cdot \sys(Y,\lambda)^{1/2} \le C
\end{equation*}
\end{prop}

\noindent The proof follows the outline discussed in the introduction. 

\vspace{3pt} 

We begin (\S \ref{subsec:standard_ellipsoids}) with a review of the geometry of standard ellipsoids $E(a,b)$ in $\C^4$, including a variant of John's theorem (Corollary \ref{cor:johns_ellipsoid}). We then present the key curvature-rotation formula (\S \ref{subsec:curvature_rotation}) and use it to bound the Ruelle invariant between two curvature integrals (Lemma \ref{lem:Ru_curvature_lowerbound}). We then prove several bounds for one of these curvature integrals in terms of diameter, area and total mean curvature (\S \ref{subsec:lower_bounding_curvature_integral}). We collect this analysis together in the final proof (\S \ref{subsec:proof_of_main_bound}).

\begin{notation} We will require the following notation throughout this section.

\begin{itemize}
	\item[(a)] $g$ is the standard metric on $\R^4$ with connection $\nabla$, and $\dvol_g = \frac{1}{2}\omega^2$ is the corresponding volume form. We also use $\langle u,v\rangle$ to denote the inner product of two vectors $u,v \in \R^4$.
	\vspace{6pt}
	\item[(b)] $\nu$ is the outward normal vector field to $Y$ and $\nu^*$ is the dual $1$-form with respect to $g$.
	\vspace{6pt}
	\item[(c)] $\sigma$ is the restriction of $g$ to $Y$ and $\dvol_\sigma$ is the corresponding metric volume form. Furthermore, $\on{area}(Y)$ denotes the surface area of $X$, i.e. the volume $\on{vol}_\sigma(Y)$ of $Y$ with respect to $\sigma$. Note that $\lambda \wedge d\lambda$ and $\dvol_\sigma$ are related (via the Liouville vector field $Z$ of $\R^4$) by
	\begin{equation}\label{eqn:volume_form_identity}
	\lambda \wedge d\lambda = \iota_Z(\frac{\omega^2}{2}|_Y) = \iota_Z(\dvol_g|_Y) =  \iota_Z(\nu^* \wedge \dvol_\sigma) = \langle Z,\nu\rangle \dvol_\sigma
	\end{equation}
	\item[(d)] $S$ is the second fundamental form of $Y$, i.e. the bilinear form given on any $u,w \in TY$ by
	\[S(u,w) := \langle \nabla_u \nu, w\rangle\]
	\item[(e)] $H$ is the mean curvature of $Y$. It is given by
	\[H := \frac{1}{3}\operatorname{trace}S\]
\end{itemize}
Note that, in this section, we will slightly abuse notation and use $\lambda$ to denote both the Liouville form $\iota_Z\omega$ and the contact form on $Y$ induced by restriction.
\end{notation}  

\subsection{Standard Ellipsoids} \label{subsec:standard_ellipsoids} Recall that a \emph{standard ellipsoid} $E(a_1,\dots,a_n) \subset \C^n$ with parameters $a_i > 0$ for $i = 1,\dots,n$ is defined as follows.
\begin{equation} \label{eqn:std_ellipsoid}
E(a_1,\dots,a_n) := \Big\{z = (z_i) \in \C^n \; : \; \sum_i \frac{\pi|z_i|^2}{a_i} \le 1\Big\}
\end{equation}
For example, $E(a) \subset \C$ is the disk of area $a$, and $E(a,\dots,a) \subset \C^n$ is the ball of radius $(a/\pi)^{1/2}$. 

\vspace{3pt}

We beginn this section with a discussion of the Riemannian and symplectic geometry of standard ellipsoids in $\C^2$. All of the relevant geometric quantities for this section can be computed explicitly in this setting. Let us record the outcome of these calculations.

\begin{lemma}[Ellipsoid Quantities] \label{lem:ellipsoid_quantities} Let $E = E(a,b)$ be a standard ellipsoid with $0 < a < b$. Then
\begin{itemize}
	\item[(a)] The diameter, surface area and volume of $E$ are given by
	\[
	\diam(E) = \frac{2}{\pi^{1/2}} \cdot b^{1/2} \qquad \area(\partial E) = \frac{4\pi^{1/2}}{3} \cdot \frac{b^2a^{1/2} - b^{1/2}a^2}{b - a} \qquad \vol{E} = \frac{ab}{2}
	\] 
	\item[(b)] The total mean curvature of $\partial E$ (i.e. the integral of the mean curvature over $\partial E$) is given by
	\[
	\int_{\partial E} H \cdot \dvol_\sigma = \frac{2\pi}{3} \cdot (b + a + \frac{ab}{b - a} \cdot \log(b/a)) 
	\]
	\item[(c)] The minimum action of a closed orbit on $\partial E$ and the systolic ratio of $\partial E$ are given by
	\[
	c(\partial E) = a \qquad \sys(\partial E) = \frac{a}{b} 
	\]
	\item[(d)] The Ruelle invariant of $\partial E$ is given by
	\[\Ru(\partial E) = a + b\]
\end{itemize}
\end{lemma}

\begin{proof} The Ruelle invariant is computed in \cite[Lem 2.1 and 2.2]{h2019}, while the minimum period of a closed orbit is computed in \cite[\S 2.1]{gh2018}. The diameter is immediate from (\ref{eqn:std_ellipsoid}). Thus, we will calculate the volume, surface area and total mean curvature. 

\vspace{3pt}

First assume that $E(a,b) = E(1,b)$ with $b \ge 1$. Let $z_j = x_j + iy_j$ be the standard coordinates on $\C^2 \simeq \R^4$. We will do all of our calculations in the following radial or toric coordinates.
\[r_i = |z_i| \qquad  \theta_i = \text{arg}(z_i) \qquad \mu_i = \pi r_i^2\]
The differential $d\mu_i$ and vector-field $\partial_{\mu_i}$ are given by
\[d\mu_i = 2\pi r_i dr_i \qquad \text{and}\qquad \partial_{\mu_i} = \frac{1}{2\pi r_i}\partial_{r_i}\]
In the $(\mu_1,\theta_1,\mu_2,\theta_2)$-coordinates, the standard metric $g$ on $\C^2$ is given by
\[\sum_i dr_i^2 + r_i^2 d\theta_i  =\sum_i \frac{1}{4\pi \mu_i} d\mu_i^2 + \frac{\mu_i}{\pi} d\theta_i^2\]
The ellipsoid $E(1,b)$ can be described as the sub-level set $F^{-1}(-\infty,1]$ of the map
\[F:\C^2 \to \R\quad F(z_1,z_2) := \pi|z_1|^2 + \frac{\pi|z_2|^2}{b} = \mu_1 + \frac{\mu_2}{b}\]
The gradient vector-field $\nabla F$ of $F$ with respect to $g$ is given by
\[\nabla F = 2\pi r_1 \partial_{r_1} + \frac{2\pi r_2}{b} \partial_{r_2} = 4\pi(\mu_1 \partial_{\mu_1} + \frac{\mu_2}{b} \partial_{\mu_2})\]
Note that the normal vector-field $\nu = \nabla F/|\nabla F|$ to $\partial E(1,b)$ can be calculated via this formula. Finally, the complement $U$ of $(\C \times 0) \cup (0 \times \C)$ in $E(1,b)$ admits the following parametrization. 
\[
\phi:(0,\infty) \times S^1 \times S^1 \to \C^4 \qquad \phi(\mu_1,\theta_1,\theta_2) = (\mu_1,\theta_1,b(1 - \mu_1),\theta_2)
\]

Now we calculate the desired quantities for $E(1,b)$. The metric volume form $\on{dvol}_g$ is given by
\[\on{dvol}_g =  d(\frac{1}{2\pi} \mu_1) \wedge d\theta_1 \wedge d(\frac{1}{2\pi} \mu_2) \wedge d\theta_2\]
Therefore, the volume may be calculated as the integral
\[
\on{vol}(E(1,b)) = \int_{E(1,b)} \frac{1}{4\pi^2} \cdot d\mu_1 \wedge d\theta_1 \wedge d\mu_2 \wedge d\theta_2 = \int_0^1 \int_0^{b(1 - \mu_1)} d\mu_1 \wedge d\mu_2 = \frac{b}{2}
\]
The area form $\on{dvol}_\sigma$ on $\partial E(1,b)$ is given by $\iota_\nu \on{dvol}_g$, which is simply
\[\on{dvol}_\sigma = \iota_\nu \on{dvol}_g = \frac{1}{|\nabla F|} \cdot \iota_{\nabla F} \on{dvol}_g\]
\[= \frac{1}{\sqrt{4\pi(\mu_1 + \mu_2/b)}} \cdot (\frac{4\pi\mu_1}{4\pi^2} \cdot d\theta_1 \wedge d\mu_2 \wedge d\theta_2 + \frac{4\pi\mu_2}{4\pi^2 b} \cdot d\mu_1 \wedge d\theta_1 \wedge d\theta_2)\]
The pullback of $\on{dvol}_\sigma$ via the map $\phi$ is given by
\[
\phi^*\on{dvol}_\sigma = \frac{1}{2\pi^{3/2}}(\mu_1 + \frac{b(1 - \mu_1)}{b^2})^{-1/2} \cdot (\mu_1 \cdot d\theta_1 \wedge d(b(1 - \mu_1)) \wedge d\theta_2 + \frac{b}{b}(1 - \mu_1) \cdot d\mu_1 \wedge d\theta_1 \wedge d\theta_2)
\]
\[
= \frac{b^{1/2}}{2\pi^{3/2}} \cdot (1 + (b-1)\mu_1)^{1/2} \cdot d\mu_1 \wedge d\theta_1 \wedge d\theta_2 
\]
Computing the surface area as the integral of $\phi^*\on{dvol}_\sigma$, we have
\[
\area(\partial E(1,b)) = \frac{b^{1/2}}{2\pi^{3/2}} \cdot \int_0^1\int_0^{2\pi}\int_0^{2\pi}  (1 + (b-1)\mu_1)^{1/2} \cdot d\theta_1 \wedge d\theta_2 \wedge d\mu_1 
\]
\[
= \frac{b^{1/2}}{2\pi^{3/2}} \cdot 4\pi^2 \cdot \int_0^1 (1 + (b-1)\mu_1)^{1/2} d\mu_1 = (4\pi b)^{1/2} \cdot \frac{2}{3(b - 1)} \cdot (1 + (b-1)\mu_1)^{3/2}|_0^1\]
\[ = \frac{4\pi^{1/2}}{3} \cdot \frac{b^2 - b^{1/2}}{b - 1}
\]
Finally, the mean curvature $H$ is given by
\[
H = \frac{1}{3|\nabla F|^3} \cdot (|\nabla F|^2 \cdot\on{tr}(\on{Hess}_F) - \on{Hess}_F(\nabla F,\nabla F)) =
\]
\[
\frac{4\pi(\mu_1 + \frac{\mu_2}{b^2}) \cdot 4\pi(1 + \frac{1}{b}) - 8\pi^2(\mu_1 + \frac{\mu_2}{b^3})}{3 \cdot (4\pi)^{3/2} \cdot (\mu_1 + \frac{\mu_2}{b^2})^{3/2}} = \frac{\sqrt{\pi}}{3} \cdot \frac{(1 + \frac{2}{b}) \cdot \mu_1 + (\frac{2}{b^2} + \frac{1}{b^3})\mu_2}{(\mu_1 + \frac{\mu_2}{b^2})^{3/2}} 
\]
The pullback of $H$ by $\phi$ is given by
\[
\phi^*H = \frac{\sqrt{\pi}}{3} \cdot \frac{(1 + \frac{2}{b}) \cdot \mu_1 + (\frac{2}{b^2} + \frac{1}{b^3}) \cdot b(1 - \mu_1)}{(\mu_1 + \frac{b(1 - \mu_1)}{b^2})^{3/2}}\]
\[
= \frac{\sqrt{\pi}}{3} \cdot \frac{\frac{2b+1}{b^2} + \frac{b^2 - 1}{b^2} \mu_1}{(1 + (1 - \frac{1}{b})\mu_1)^{3/2}} = \frac{\pi^{1/2}}{3 b^{1/2}} \cdot \frac{(2b+1) + (b^2 - 1) \mu_1}{(1 + (b- 1)\mu_1)^{3/2}} 
\]
Computing the mean curvature as the integral of $\phi^*(H \cdot \on{dvol}_\sigma)$, we have
\[
\int_{\partial E(1,b)} H \on{dvol}_\sigma = \frac{1}{6\pi} \cdot \int_0^1\int_0^{2\pi}\int_0^{2\pi} \frac{(2b + 1) + (b^2 - 1)\mu_1}{1 + (b-1)\mu_1} \cdot d\theta_1 \wedge d\theta_2 \wedge d\mu_1 
\]
\[
=  \frac{2\pi}{3} \cdot \int_0^1\frac{(2b + 1) + (b^2 - 1)\mu_1}{1 + (b-1)\mu_1} \cdot d\mu_1 = \frac{2\pi}{3} (b + 1 + \frac{b}{b-1} \cdot \log(b))
\]

To deduce the general case of the computation from this special case note that, if $U$ is any smooth domain, then
\begin{equation} \label{eqn:geometric_scaling}
\vol{\lambda \cdot U} = \lambda^4 \cdot \vol{U} \quad \area(\lambda \cdot U) = \lambda^3 \cdot \area(U) \quad \int_{\lambda \cdot \partial U} H \cdot \on{dvol}_\sigma = \lambda^2 \cdot \int_{\partial U} H \cdot \on{dvol}_\sigma
\end{equation}
Any ellipsoid $E(a,b)$ can be scale so to an ellipsoid with $a = 1$, since
\[
\lambda \cdot E(a,b) = E(\lambda^2 a,\lambda^2 b) \quad\text{and thus}\quad E(a,b) = a^{1/2} \cdot E(1,b/a)
\]
The general case now follows from the special case and the scaling properties (\ref{eqn:geometric_scaling}). \end{proof}

Any convex boundary in $\R^{2n}$ can be sandwiched between a standard ellipsoid and a scaling of that ellipsoid by a factor of $2n$, after the application of an affine symplectomorphism. To see this, first recall the following well-known result of John. 

\begin{thm}[John Ellipsoid] \cite{j1948} \label{thm:john_ellipsoid} Let $K \subset \R^n$ be a convex domain. Then there exists a unique ellipsoid $E$ of maximal volume in $K$. Furthermore, if $c \in X$ is the center of $E$ then
\[E \subset K \subset c + n(E - c)\]
\end{thm}

\noindent Any ellipsoid $E$ is carried to a standard ellipsoid $E(a,b)$ by some affine symplectomorphism $T$. Furthermore, note that we have the following elementary result, which can be demonstrated using a Moser argument. 

\begin{lemma} \label{lem:strict_contactomorphism} Let $\phi:(Y,\lambda) \to (Y',\lambda')$ be a diffeomorphism such that $\phi^*\lambda' = \lambda + df$. Then $\phi$ is isotopic to a strict contactomorphism.\end{lemma}

\noindent Since $\R^{2n}$ is contractible, $T^*\lambda = \lambda + df$ automatically on $\R^{2n}$. Thus, $T$ carries any star-shaped hypersurface $Y = \partial X$ to a strictly contactomorphic $T(Y)$ by Lemma \ref{lem:strict_contactomorphism}, and we conclude the following result.

\begin{cor} \label{cor:johns_ellipsoid} Let $X \subset \R^{2n}$ be a convex star-shaped domain with boundary $Y$. Then $Y$ is strictly contactomorphic to the boundary $\partial K$ of a convex domain $K$ with $E(a_1,\dots,a_n) \subset K \subset 4 \cdot E(a_1,\dots,a_n)$.
\end{cor}

When a convex domain in $\R^4$ is squeezed between an ellipsoid and its scaling, we can estimate many important geometric quantities of $X$ in terms of the ellipsoid itself.

\begin{lemma} \label{lem:quantity_sandwich} Let $X \subset \R^4$ be a convex domain with smooth boundary $Y$ such that
\begin{equation} \label{eqn:ellipsoid_sandwich_property} E(a,b) \subset X \subset c \cdot E(a,b) \qquad \text{ for some }b \ge a > 0 \text{ and }c \ge 0\end{equation}
Then there is a constant $C > 0$ dependent only on $c$ such that
\begin{equation} \label{eqn:quantity_sandwich_a}
b^{1/2} \le \diam(X) \le C \cdot b^{1/2} \qquad ba^{1/2} \le \area(Y) \le C \cdot ba^{1/2}
\end{equation}
\begin{equation} \label{eqn:quantity_sandwich_b}
b \le \int_Y H \cdot \dvol_\sigma \le C \cdot b \qquad \frac{ab}{2} \le \vol{X} \le C \cdot ab
\end{equation}
\begin{equation} \label{eqn:quantity_sandwich_c}
a \le c(X) \le C \cdot a \qquad C^{-1} \cdot \frac{a}{b} \le \sys(Y) \le C \cdot \frac{a}{b}
\end{equation}
\end{lemma}

\begin{remark} The optimal constants in the estimates (\ref{eqn:quantity_sandwich_a})-(\ref{eqn:quantity_sandwich_c}) are not important to the arguments below. They could be explicitly computed in the following proof. \end{remark}

\begin{proof} First, note that $c \cdot E(a,b)$ is also a standard ellipsoid. More precisely, we know that
\[c \cdot E(a,b) = E(c^2 \cdot a, c^2 \cdot b)\]
We now derive the desired estimates from Lemma \ref{lem:ellipsoid_quantities} and the monotonicity of the relevant quantities under inclusion of convex domains. 

\vspace{3pt}

The diameter $\diam(X)$ and volume $\vol{X}$ are monotonic with respect to inclusion of arbitrary open subsets, and so from Lemma \ref{lem:ellipsoid_quantities}(a) we acquire
\[b^{1/2} \le \on{diam}(X) \le \frac{2c}{\pi^{1/2}} \cdot b^{1/2} \qquad\text{and}\qquad \frac{ab}{2} \le \vol{X} \le \frac{c^4}{2} \cdot ab\]
The surface area and total mean curvature are monotonic with respect to inclusion of convex domains, since
\[\int_Y H \dvol_\sigma = 4 \cdot V_2(X) \qquad\text{and}\qquad \area(Y) = 4 \cdot V_3(X)\]
Here $V_i(X)$ is the \emph{$i$th cross-sectional measure} \cite[\S 19.3]{bz1980}, which is monotonic with respect to inclusions of convex domains by \cite[p.138, Equation 13]{bz1980}. Furthermore, when $0 < a < b$ (and in the limit as $b \to a$), one may verify that
\begin{equation}\label{eqn:quantity_sandwich_ab_estimates_a}
ba^{1/2} \le \frac{b^2a^{1/2} - b^{1/2}a^2}{b - a} \le \frac{3}{2} \cdot ba^{1/2} \qquad\text{and}\qquad b \le b + a + \frac{ab}{b - a} \cdot \log(b/a) \le 3b\end{equation}
Thus, by applying the monotonicity property, (\ref{eqn:quantity_sandwich_ab_estimates_a}) and Lemma \ref{lem:ellipsoid_quantities}(a)-(b), we have
\[
\frac{4\pi^{1/2}}{3} \cdot ba^{1/2} \le \area(Y) \le \frac{4\pi^{1/2}}{3}c^3 \cdot (\frac{3}{2}ba^{1/2}) \qquad\text{and}\qquad \frac{2\pi}{3} \cdot b \le \int_Y H \cdot \dvol_\sigma \le \frac{2\pi}{3}c^2 \cdot 3b\] 
Finally, the minimum orbit length $c(X)$ coincides with the 1st Hofer-Zehnder capacity $c^{HZ}_1(X)$ on convex domains, and is thus monotonic with respect to symplectic embeddings. Thus by Lemma \ref{lem:ellipsoid_quantities}(a) and (c), we have
\[
a \le c(X) \le c^2 \cdot a \qquad\text{and}\qquad c^{-4} \cdot \frac{a}{b} \le \frac{c(X)^2}{2\vol{X}} = \sys(Y) \le c^4 \cdot \frac{a}{b} 
\]
This concludes the proof, after choosing $C$ larger than the constants appearing above. \end{proof} 

\subsection{Curvature-Rotation Formula} \label{subsec:curvature_rotation} Identify $\R^4$ with the quaternions $\mathbb{H}^1$ via
\[\R^4\ni (x_1,y_1,x_2,y_2)\mapsto x_1 + y_1 I + x_2 J + y_2 K\in \mathbb{H}^1\]
This equips $\R^4$ with a triple of complex structures.
\[I:T\R^4 \to T\R^4 \qquad J:T\R^4 \to T\R^4 \qquad K:T\R^4 \to T\R^4\]
The Reeb vector-field $R$ of the contact form on a star-shaped hypersurface $Y$ is parallel to $I$ applied to the normal vector-field $\nu$ to $Y$. Precisely, we have
\begin{equation} \label{eqn:Reeb_vector_vs_Inu} R = \frac{I\nu}{\langle Z,\nu\rangle}\end{equation}
We can utilize these structures to formulate an explicit representative of the standard homotopy class of trivialization $\tau:\xi \simeq \R^2$ on the contact structure $\xi$ on the boundary $Y$ of the convex star-shaped domain $X$ (or more generally, on any star-shaped boundary).
\begin{definition} \label{def:quaternion_trivialization} The \emph{quaternion trivialization} $\tau:\xi \simeq Y \times \C$ is the symplectic trivialization given by
\[
\tau:\xi \xrightarrow{\pi} Q \xrightarrow{q^{-1}} Y \times \C
\]
Here $Q \subset TY$ is the symplectic sub-bundle $\text{span}(J\nu,K\nu)$, $\pi:\xi \to Q$ is the projection map from $\xi$ to $Q$ along the Reeb direction, and $q:Y \times \C \to Q$ is the bundle map given on $z = a + ib$ by
\begin{equation}
q_p(z) := z \cdot J\nu_p = (a + Ib) \cdot J\nu_p = aJ\nu_p + bK\nu_p\end{equation}
\end{definition}

The key property of the quaternion trivialization is the following relation of the rotation density (see Definition \ref{def:rotation_density}) to extrinsic curvature, originally due to Ragazzo-Salom\~{a}o (c.f. \cite{rs2006}).

\begin{prop}[Curvature-Rotation] \label{prop:rotation_curvature_formula} \cite[Prop 4.7]{ch2020} Let $\tau$ be the quaternion trivialization on the contact structure $\xi$ of $Y \subset \R^4$. Then
\begin{equation} \label{eqn:curvature_rotation}
\varrho_\tau(y,s) = \frac{1}{2\pi \cdot \langle Z_y,\nu_y\rangle}(S(I\nu_y,I\nu_y) + S(e^{2\pi i s} \cdot J\nu_y, e^{2\pi i s} \cdot J\nu_y))\end{equation}
\end{prop}

\noindent Note that this result holds for any star-shaped boundary, not only convex ones.

\vspace{3pt}

As an easy consequence of (\ref{eqn:curvature_rotation}), we have the following bound on the Ruelle invariant of $Y$.

\begin{lemma} \label{lem:Ru_curvature_lowerbound} The Ruelle invariant $\Ru(Y)$ is bounded by the following curvature integrals.
\begin{equation} \label{eqn:Ru_curvature_bound_main}
\frac{1}{2\pi} \cdot \int_Y S(I\nu,I\nu) \dvol_\sigma \le \Ru(Y) \le \frac{3}{2\pi} \cdot \int_Y H \dvol_\sigma
\end{equation}
\end{lemma}

\begin{proof} By Lemma \ref{lem:Ru_alternate_formula}, we have the following integral formula for the Ruelle invariant.
\begin{equation} \label{eqn:Ru_curvature_lowerbound_0}
\Ru(Y) = \lim_{T \to \infty} \frac{1}{T} \int_0^T \left(\int_Y [\bar{\Phi}_t^*\varrho_\tau](-,s) \cdot \lambda \wedge d\lambda \right) dt
\end{equation}
By the curvature-rotation formula in Proposition \ref{prop:rotation_curvature_formula}, we can write the integrand as 
\begin{equation} \label{eqn:Ru_curvature_lowerbound_a}
[\bar{\Phi}_t^*\varrho_\tau](-,s) = \bar{\Phi}_t^*\Big(\frac{1}{2\pi\cdot \langle Z,\nu\rangle}(S(I\nu,I\nu) + S(e^{2\pi i s} \cdot J\nu, e^{2\pi i s}\cdot J\nu))\Big)
\end{equation}

To bound the righthand side of (\ref{eqn:Ru_curvature_lowerbound_a}), note that $I\nu,e^{2\pi i s}\cdot J\nu$ and $e^{2\pi i s} \cdot K\nu$ form an orthonormal basis of $TY$ with respect to the restricted metric $g|_Y$, so that 
\[S(I\nu,I\nu) + S(e^{2\pi i s}\cdot J\nu,e^{2\pi i s}\cdot J\nu) + S(e^{2\pi i s}\cdot K\nu,e^{2\pi i s}\cdot K\nu) = \on{trace}(S) = 3H \]
Furthermore, since $Y$ is convex, the second fundamental form $S$ is positive semi-definite. Therefore by (\ref{eqn:Ru_curvature_lowerbound_a}), we have the following lower and upper bound.
\begin{equation} \label{eqn:Ru_curvature_lowerbound_b}
\bar{\Phi}_t^*\Big(\frac{S(I\nu,I\nu)}{\langle Z,\nu\rangle}\Big) \le [\bar{\Phi}_t^*\varrho_\tau](-,s) \le 3 \cdot \bar{\Phi}_t^*\Big(\frac{H}{\langle Z,\nu\rangle}\Big)
\end{equation} 
It is key here that the lower and upper bounds in (\ref{eqn:Ru_curvature_lowerbound_b}) are independent of $s$. To simplify the two sides of (\ref{eqn:Ru_curvature_lowerbound_b}), let $F:Y \times S^1 \to \R$ be any map pulled back from a map $F:Y \to \R$. Since the flow $\bar{\Phi}_t$ on $Y \times S^1$ lifts the Reeb flow $\phi_t$ on $Y$, and $\phi_t$ preserves $\lambda$, we have
\[\bar{\Phi}_t^*\Big(\frac{F}{\langle Z,\nu\rangle}\Big) \cdot \lambda \wedge d\lambda = \phi_t^*\Big(\frac{F}{\langle Z,\nu\rangle}\Big) \cdot \lambda \wedge d\lambda = \phi_t^*\Big(F \cdot \frac{\lambda \wedge d\lambda}{\langle Z,\nu\rangle}\Big) = \phi^*_t\Big(F \cdot \on{dvol}_\sigma\Big)
\]
Since the integral of $\phi_t^*(F \cdot \dvol_\sigma)$ over $Y$ is independent of $t$, we have
\begin{equation} \label{eqn:Ru_curvature_lowerbound_c}
\frac{1}{T} \int_0^T \left(\int_Y \bar{\Phi}_t^*\Big(\frac{F}{\langle Z,\nu\rangle}\Big) \cdot \lambda \wedge d\lambda \right) dt = \frac{1}{T} \int_0^T \left(\int_Y F \cdot \dvol_\sigma \right) dt = \int_Y F \cdot \dvol_\sigma
\end{equation}

By plugging in the estimate (\ref{eqn:Ru_curvature_lowerbound_b}) to the integral formula (\ref{eqn:Ru_curvature_lowerbound_0}) and applying (\ref{eqn:Ru_curvature_lowerbound_c}) to the functions $S(I\nu,I\nu)$ and $H$ on $Y$, we acquire the desired bound (\ref{eqn:Ru_curvature_bound_main}).
\end{proof}

\subsection{Bounding Curvature Integrals} \label{subsec:lower_bounding_curvature_integral} We now further simplify the lower bound of the Ruelle invariant in Lemma \ref{lem:Ru_curvature_lowerbound} by estimating (from below) the integral
\[\int_Y S(I\nu,I\nu) \cdot \dvol_\sigma\]
using the geometric quantities (e.g. area and diameter) appearing in \S \ref{subsec:standard_ellipsoids}. This will help us to leverage the sandwich estimates in Lemma \ref{lem:quantity_sandwich} in the proof of the Ruelle invariant bound in \S \ref{subsec:proof_of_main_bound}.

\vspace{3pt}

Recall that $X \subset \R^4$ denotes a convex domain with smooth boundary $Y$. Let $\psi:\R \times Y \to Y$ be the flow by $I\nu$. Let $S_T$ and $H_T$ denote the time-averaged versions of $S(I\nu,I\nu)$ and $H$, respectively.
\begin{equation}
S_T := \frac{1}{T} \int_0^T S(I \nu,I \nu) \circ \psi_t  dt \qquad H_T := \frac{1}{T} \int_0^T H \circ \psi_t dt\end{equation}
We will also need to consider a time-averaged acceleration function $A_T$ on $Y$. Namely, let $\gamma:\R \to Y$ be a trajectory of $I\nu$ with $\gamma(0) = x$. Then we define
\begin{equation}A_T := \frac{1}{T} \int_0^T |\nabla_{I\nu}I\nu| \circ \psi_t dt \qquad\text{or equivalently}\qquad A_T(x) = \frac{1}{T} \int_0^T |\ddot{\gamma}| dt\end{equation}

The first ingredient to the bounds in this section is the following estimate relating these three time-averaged functions.

\begin{lemma} \label{lem:acceleration_bound} For any $T > 0$, the functions $A_T,H_T$ and $S_T$ satisfy $A_T^2 \le 3 \cdot H_T \cdot S_T$ pointwise. 
\end{lemma}

\begin{proof} In fact, the non-time-averaged version of this estimate holds. We will now show that
\begin{equation} \label{eqn:SIv_lower_bound_step_1}
|\nabla_{I\nu} I\nu|^2 \le 3H \cdot S(I\nu,I\nu)
\end{equation}
To start, we need a formula for $\nabla_{I\nu} I\nu$ in terms of the second fundamental form, as follows.
\[
\nabla_{I\nu}I\nu = \langle \nu,\nabla_{I\nu}I\nu\rangle \nu + \langle I\nu,\nabla_{I\nu}I\nu\rangle I\nu +  \langle J\nu,\nabla_{I\nu}I\nu\rangle J\nu  +  \langle K\nu,\nabla_{I\nu}I\nu\rangle K\nu 
\]
\[
= -\langle I\nu,\nabla_{I\nu}\nu\rangle \nu - \langle I^2\nu,\nabla_{I\nu}\nu\rangle I\nu - \langle IJ\nu ,\nabla_{I\nu}\nu\rangle J\nu - \langle IK\nu,\nabla_{I\nu} \nu\rangle K\nu
\]
Applying the quaternionic relations $I^2 = -1$, $IJ = K$ and $IK = -J$, we can rewrite this as
\[
-\langle I\nu,\nabla_{I\nu}\nu\rangle \nu + \langle\nu,\nabla_{I\nu}\nu\rangle I\nu - \langle K\nu ,\nabla_{I\nu}\nu\rangle J\nu + \langle J\nu,\nabla_{I\nu} \nu\rangle K\nu\]
Finally, applying the definition of the second fundamental form we find that
\[\nabla_{I\nu} I\nu = -S(I\nu,I\nu)\nu - S(I\nu,K\nu)J\nu + S(I\nu,J\nu)K\nu
\]
To estimate the righthand side, we note that $S(u,v)^2 \le S(u,u)S(v,v)$ for any vectorfields $u$ and $v$ by Cauchy-Schwarz, since $S$ is positive semi-definite. Thus we have
\[
|\nabla_{I\nu} I\nu|^2 \le S(I\nu,I\nu)^2 + S(I\nu,I\nu)S(J\nu,J\nu) + S(I\nu,I\nu)S(K\nu,K\nu) = 3H \cdot S(I\nu,I\nu)
\] 
This proves (\ref{eqn:SIv_lower_bound_step_1}) and the desired estimate follows immediately by Cauchy-Schwarz.
\begin{equation}\label{eqn:SIv_lower_bound_step_2a} A_T^2 = \big(\frac{1}{T}\int_0^T |\nabla_{I\nu} I\nu| \circ \psi_t dt \big)^2 \le 3 \cdot \frac{1}{T} \int_Y H \circ \psi_t dt \cdot \frac{1}{T} \int_Y S(I\nu,I\nu) \circ \psi_t dt = 3 H_T \cdot S_T\end{equation}
This concludes the proof of the lemma. \end{proof}

As a consequence, we get the following estimate for the curvature integral of interest in terms of area, total mean curvature and the time-averaged acceleration $A_T$.

\begin{lemma} \label{lem:acceleration_bound} Let $\Sigma \subset Y$ be an open subset of $Y$ and let $T > 0$. Then 
\begin{equation} \label{eqn:acceleration_bound} 
\int_Y S(I\nu,I\nu) \cdot \dvol_\sigma \ge \frac{\area(\Sigma)^2}{3 \cdot \int_Y H \dvol_\sigma} \cdot \on{min}_\Sigma(A_T)^2 \end{equation}
\end{lemma}

\begin{proof} We first note that $I\nu$ preserves the volume form $\dvol_\sigma$, since
\[\mathcal{L}_{I\nu}(\dvol_\sigma) = d\iota_{I\nu} \dvol_\sigma = d\iota_R(\lambda \wedge d\lambda) = d^2\lambda = 0\]
Here $R$ is the Reeb vector-field on $Y$, and the above equalities follow from (\ref{eqn:Reeb_vector_vs_Inu}) and (\ref{eqn:volume_form_identity}). Thus, time-averaging leaves the integral over $Y$ unchanged.
\[
\int_Y H_T \dvol_\sigma = \int_Y H \dvol_\sigma \qquad\text{and}\qquad \int_Y S_T \dvol_\sigma = \int_Y S(I\nu,I\nu) \dvol_\sigma
\]
We can thus integrate the estimate $A_T^2 \le 3 H_T \cdot S_T$ to see that
\[
\on{min}(A_T)^2 \cdot \area(\Sigma)^2 \le \Big(\int_\Sigma A_T \cdot \dvol_\sigma\Big)^2 \le \Big(\sqrt{3} \cdot \int_\Sigma H_T^{1/2} \cdot S_T^{1/2} \cdot \dvol_\sigma\Big)^2\]
\[\le 3 \cdot \int_\Sigma H_T \cdot \dvol_\sigma \cdot \int_\Sigma S_T \cdot \dvol_\sigma \le 3 \cdot \int_Y H \cdot \dvol_\sigma \cdot \int_Y S(I\nu,I\nu) \cdot \dvol_\sigma
\]
After some rearrangement, this is the desired estimate. \end{proof}

Every quantity on the righthand side of (\ref{eqn:acceleration_bound}) can be controlled using the estimates in Lemma \ref{lem:quantity_sandwich}, with the exception of the term involving the time-averaged acceleration $A_T$. However, we can bound $A_T$ in terms of $\diam(X)^{-1}$, using the following general fact about curves of unit speed.

\begin{lemma} \label{lem:AT1} Let $\gamma:[0,\infty) \to Y$ be a curve with $|\dot{\gamma}| = 1$ and let $C$ satisfy $0 < C < 1$. Then 
\[\frac{1}{T}\int_0^T |\ddot{\gamma}| dt \ge \frac{C}{\on{diam}(X)} \qquad\text{ for all }T \gg 0\]
\end{lemma}

\begin{proof} Let $T$ satisfy $T > CT + 2 \cdot \diam(Y)$. Then by Cauchy-Schwarz, we have 
\begin{equation}
\label{eqn:SIv_lower_bound_step_2c}
\diam(X)\int_0^T |\ddot{\gamma}| dt \ge \int_0^T |\gamma| \cdot |\ddot{\gamma}| dt \ge \int_0^T |\langle \ddot{\gamma},\gamma\rangle| dt \ge \Big|\int_0^T \langle \ddot{\gamma},\gamma\rangle dt \Big|
\end{equation}
On the other hand, by integration by parts we acquire
\begin{equation}
\label{eqn:SIv_lower_bound_step_2d}
\Big|\int_0^T \langle \ddot{\gamma},\gamma\rangle dt \Big| \ge \Big|\int_0^T |\dot{\gamma}|^2 dt - \langle \gamma,\dot{\gamma}\rangle|_0^T \Big| \ge T - 2\diam(X) \ge CT
\end{equation}
Combining the estimates (\ref{eqn:SIv_lower_bound_step_2c}) and (\ref{eqn:SIv_lower_bound_step_2d}) yields the claimed bound. \end{proof}

In particular, Lemma \ref{lem:AT1} implies that $A_T \ge C \cdot \diam(X)^{-1}$ for all $C < 1$ and sufficiently large $T$. Combining this with Lemma \ref{lem:acceleration_bound} and taking $C \to 1$, we acquire the following corollary.

\begin{cor} \label{cor:SIv_lower_bound} Let $X \subset \R^4$ be a convex star-shaped domain with boundary $Y$. Then
\begin{equation}\label{eqn:SIv_lower_bound}
\int_Y S(I\nu,I\nu) \dvol_\sigma \ge \frac{\area(Y)^2}{3 \cdot\diam(X)^2 \cdot \int_Y H \dvol_\sigma}
\end{equation}
\end{cor} 
\noindent At this point, we can already apply Lemma \ref{lem:quantity_sandwich} to derive a uniform lower bound for $\ru(Y) \cdot \sys(Y)^{-1/2}$. However, this inequality does not have the desired exponent for $\on{sys}$. In order to fix this, we must derive a different estimate similar to Corollary \ref{cor:SIv_lower_bound} when $\on{sys}(Y)$ is near $0$. This is the objective of the rest of this part.

\vspace{3pt}

We will also need a less crude estimate on the time-averaged acceleration that uses the geometry of vector-field $I\nu$, but requires the hypothesis that $X$ has small systolic ratio.

\begin{lemma} \label{lem:AT_bound_for_small_sys} Suppose that $X$ satisfies $E(a,b) \subset X \subset 4 \cdot E(a,b)$ and let $\Sigma \subset Y$ be the open subset
\[\Sigma = Y \cap \C \times \on{int}(E(b/2))\]
Then there is an $\epsilon > 0$ and a $C > 0$ independent of $a,b$ and $X$ such that, if $a/b < \epsilon$ and $T = b^{1/2}$, then
\[A_T \ge C \cdot a^{-1/2} \quad \text{on}\quad \Sigma \qquad\text{and}\qquad \area(\Sigma) \ge C \cdot \area(Y)\]
\end{lemma}

\begin{proof} To bound $A_T$, the strategy is to show that the projection of $I\nu$ to the 2nd $\C$-factor is bounded along $\Sigma$ by $(a/b)^{1/2}$. Thus, a length $T = b^{1/2}$ trajectory $\gamma$ of $I\nu$ stays within a ball of diameter roughly $a^{1/2}$, and a variation of Lemma \ref{lem:AT1} implies the desired bound. 

\vspace{3pt}

To bound $\area(\Sigma)$, the strategy is (essentially) to use the monotonicity of area under the inclusion $E(a,b) \subset X$ to reduce to the case of an ellipsoid. We can then use the estimates in Lemmas \ref{lem:ellipsoid_quantities} and \ref{lem:quantity_sandwich} to deduce the result.   

\vspace{3pt}

{\bf Projection Bound.} Let $\pi_j:\R^4 \simeq \C^2 \to \C$ denote the projections to each $\C$-factor for $j = 1,2$. We begin by noting that there is an $A > 0$ independent of $X,a$ and $b$ such that
\begin{equation} \label{eqn:SIv_bound_for_small_sys_a}
|\pi_2 \circ I\nu(x)| = |\pi_2 \circ \nu(x)| < A \cdot (a/b)^{1/2} \qquad\text{if}\qquad x \in Y \quad \text{and} \quad \pi_2(x) \in E(3b/4)\end{equation}
To deduce (\ref{eqn:SIv_bound_for_small_sys_a}), assume that $x \in Y$ satisfies $\pi_2(x) \in E(3b/4)$ and that $\pi_2 \circ \nu(x) \neq 0$. Let $z \in 0 \times \partial E(b)$ be the unique vector such that $\pi_2(z-x)$ is a positive scaling of $\pi_2(\nu(x))$. Note that $z \in X$ since
\[0 \times E(b) \subset E(a,b) \subset X\]
Furthermore, since $X$ is convex, we know that $\langle \nu(x),w - x\rangle \le 0$ for any $w \in X$. Therefore 
\begin{equation}
0 \ge \langle \nu(x),z - x\rangle = |\pi_2 \circ \nu(x))| \cdot |\pi_2(z - x)|  + \langle \pi_1 \circ \nu(x), \pi_1(z - x)\rangle\end{equation}
Now note that since $\pi_2(x) \in E(3b/4)$ and $\pi_2(z) \in \partial E(b)$, we know that
\[|\pi_2(z - x)| \ge \frac{1 - (3/4)^{1/2}}{\pi^{1/2}} \cdot b^{1/2}\]
Likewise, $\pi_1(X) \subset 4 \cdot E(a)$ so that $|\pi_1(z - x)| \le 4a^{1/2}/\pi^{1/2}$. Finally, $|\pi_1 \circ \nu(x)| \le |\nu(x)| = 1$.  Thus, we can conclude that
\[|\pi_2 \circ \nu(x)| \le \frac{|\pi_1 \circ \nu(x)| \cdot |\pi_1(z - x)|}{|\pi_2(z - x)|} \le \frac{4}{1 - (3/4)^{1/2}} \cdot (a/b)^{1/2}\]

{\bf Acceleration Bound.} Now let $T = b^{1/2}$ and let $\gamma:[0,T] \to Y$ be a trajectory of $I\nu$ with $\gamma(0) \in \Sigma$. Since $\pi_2(\gamma(0)) \in E(b/2)$, we know that there is an interval $[0,S] \subset [0,T]$ where $\pi_2 \circ \gamma([0,S]) \subset E(3b/4)$. Thus, by (\ref{eqn:SIv_bound_for_small_sys_a}), we know that for $t \in [0,S]$ we have
\begin{equation} \label{eqn:SIv_bound_for_small_sys_b}
|\pi_2(\gamma(t) - \gamma(0))| \le \int_0^t |\pi_2 \circ I\nu \circ \gamma| dt \le A \cdot (a/b)^{1/2} \cdot t \le A \cdot a^{1/2}\end{equation}
By picking $\epsilon > 0$ small enough so that $a/b$ is small, we can ensure the following inequality.
\begin{equation} \label{eqn:SIv_bound_for_small_sys_d} A a^{1/2} \le (\frac{3b}{4\pi})^{1/2} - (\frac{b}{2\pi})^{1/2}\end{equation}
With this choice of $\epsilon$, (\ref{eqn:SIv_bound_for_small_sys_b}) and (\ref{eqn:SIv_bound_for_small_sys_d}) imply that $\pi_2(\gamma(t) - \gamma(0)) \in E(3b/4)$ if $0 \le t \le T$. In fact, (\ref{eqn:SIv_bound_for_small_sys_b}) implies that $\gamma$ is inside of a ball, i.e.
\[
\gamma(t) \in E(16a) \times E(\pi A^2 \cdot a) + p \subset B \cdot E(a,a) + p \quad\text{where}\quad p := 0 \times \pi_2(\gamma(0))
\]
Here $B := (16+\pi A^2)^{1/2}$. The diameter of the ball $B \cdot E(a,a)$ is $2B \cdot (a/\pi)^{1/2}$. Therefore, by applying (\ref{eqn:SIv_lower_bound_step_2c}) and (\ref{eqn:SIv_lower_bound_step_2d}) we see that
\[
\frac{2Ba^{1/2}}{\pi^{1/2}} \cdot A_T(x) = \frac{\diam(B \cdot E(a,a))}{T}\cdot \int_0^T |\ddot{\gamma}| dt \ge 1 - \frac{2 \diam(B \cdot E(a,a))}{T} = 1 - \frac{4B}{\pi^{1/2}} \cdot (a/b)^{1/2}
\]
We now choose $C > 0$ and $\epsilon > 0$ independent of $a,b$ and $X$, such that
\[A_T(x) \ge (\frac{\pi^{1/2}}{2B} - 2 \cdot (a/b)^{1/2}) \cdot a^{-1/2} \ge C a^{-1/2} \quad\text{if}\quad a/b \le \epsilon \]
This proves the desired bound on time-averaged acceleration.

\vspace{3pt}

{\bf Area Bound.} Let $U$ denote the convex domain given by the intersection $X \cap (\C \times E(b/2))$. Note that we have the following inclusion.
\[
E(a/2,b/2) \subset E(a,b) \cap (\C \times E(b/2)) \subset U
\]
Furthermore, the boundary of $U$ decomposes as follows.
\[
\partial U = \Sigma \cup \Sigma' \qquad\text{where}\qquad \Sigma' := X \cap (\C \times \partial E(b/2))
\]
Since $X \subset 4 \cdot E(a,b)$, we have $\Sigma' \subset R$ where $R$ is the hypersurface
\[
R := 4 \cdot E(a,b) \cap (\C \times \partial E(b/2)) = E(31a/2) \times \partial E(b/2)
\]
Combining the above facts and applying the monotonicity of surface area under inclusion of convex domains, we find that
\[
\area(\Sigma) = \area(\partial U) - \area(\Sigma') \ge \area(\partial E(a/2,b/2)) - \area(R)
\]
By Lemma \ref{lem:quantity_sandwich} and direct calculation, we compute the areas of $\partial E(a/2,b/2)$ and $R$ to be
\[
\area(\partial E(a/2,b/2)) \ge 2^{-3/2} \cdot ba^{1/2} \qquad \area(R) = \frac{31a}{2} \cdot (2\pi b)^{1/2} = 31 \cdot (\pi/2)^{1/2} \cdot (a/b)^{1/2} \cdot ba^{1/2}
\]
Now let $B < 2^{-5/2}$ and choose $\epsilon > 0$ small enough to that if $a/b < \epsilon$ then
\[2^{-3/2} - 31 \cdot (\pi/2)^{1/2} \cdot (a/b)^{1/2} > B\]
By applying this inequality and the upper bound for area in Lemma \ref{lem:quantity_sandwich}, we find that for some $C > 0$ independent of $X,a$ and $b$ and an $\epsilon > 0$ as above, we have
\[\area(\Sigma) \ge (2^{-3/2} - 31 \cdot (\pi/2)^{1/2} \cdot (a/b)^{1/2}) \cdot ba^{1/2} \ge C \cdot ba^{1/2} \ge \area(Y)\]
This yields the desired area bound and concludes the proof of the lemma. \end{proof}

By plugging the bounds for $A_T$ and $\area(\Sigma)$ from Lemma \ref{lem:AT_bound_for_small_sys} into Lemma \ref{lem:acceleration_bound}, we acquire the following variation of Corollary \ref{cor:SIv_lower_bound}.

\begin{cor} \label{cor:SIv_bound_for_small_sys} Let $X$ be a convex domain with smooth boundary $Y$, such that $E(a,b) \subset X \subset 4 \cdot E(a,b)$. Then there exists a $C > 0$ and $\epsilon > 0$ independent of $X,a$ and $b$ such that
\[\int_Y S(I\nu,I\nu) \cdot \dvol_\sigma \ge C \cdot \frac{\area(Y)^2}{a \cdot \int_Y H \dvol_\sigma}  \quad \text{if} \quad a/b < \epsilon\]
\end{cor}

\subsection{Proof Of Main Bound} \label{subsec:proof_of_main_bound} We now combine the results of \S \ref{subsec:standard_ellipsoids}-\ref{subsec:lower_bounding_curvature_integral} to prove Proposition \ref{prop:Ru_bound}.

\begin{proof} (Proposition \ref{prop:Ru_bound}) By Lemma \ref{cor:johns_ellipsoid}, we may assume that $X$ is sandwiched between standard ellipsoid $E(a,b)$ with $0 < a \le b$ and a scaling.
\[E(a,b) \subset X \subset 4 \cdot E(a,b)\]
We begin by proving the lower bound, under this assumption. By Lemma \ref{lem:Ru_curvature_lowerbound}, we have 
\begin{equation}
\label{eqn:SIv_bound_large_sys}
\Ru(Y) \ge \frac{1}{2\pi} \cdot \int\limits_{Y} S(I\nu,I\nu) \dvol_\sigma 
\end{equation}
By applying the lower bound in Corollary \ref{cor:SIv_lower_bound} and using the estimates for diameter, area, total curvature, volume and systolic ratio in Lemma \ref{lem:quantity_sandwich}, we see that for constants $B,C > 0$, we have
\begin{equation} \label{eqn:SIv_bound_big_sys}
\int\limits_{Y} S(I\nu,I\nu) \cdot \dvol_\sigma \ge \frac{\area(Y)^2}{3 \cdot\diam(Y)^2 \cdot \int_Y H \dvol_\sigma} \ge B \cdot a \ge C \cdot \vol{X}^{1/2} \cdot \sys(Y)^{1/2}
\end{equation}
On the other hand, suppose that $\frac{a}{b}\ll 1$. Due to Lemma \ref{lem:quantity_sandwich}, this is equivalent to $\sys(Y)\ll 1$. By Corollary \ref{cor:SIv_bound_for_small_sys} and the estimates in Lemma \ref{lem:quantity_sandwich}, there are constants $A,B,C > 0$ with
\begin{equation}
\label{eqn:SIv_bound_small_sys}
\int\limits_{Y} S(I\nu,I\nu) \dvol_\sigma
\ge A\cdot\frac{\area(Y)^2}{a \cdot \int_Y H \dvol_\sigma}
\ge B \cdot b
\ge C\cdot \vol{X}^{1/2}\cdot \sys(Y)^{-1/2}
\end{equation}
By assembling the estimate (\ref{eqn:SIv_bound_large_sys}) with the two estimates \eqref{eqn:SIv_bound_big_sys} and \eqref{eqn:SIv_bound_small_sys}, we deduce the following lower bound for some $C > 0$.
\begin{equation}
\Ru(Y)\ge C\cdot \vol{X}^{1/2}\cdot \sys(Y)^{-1/2}
\end{equation}
After some rearrangement, this is the desired lower bound. 

\vspace{3pt}

The second inequality is easier to show. By using the upper bound in Lemma \ref{lem:Ru_curvature_lowerbound} and the estimate for the mean curvature in Lemma \ref{lem:quantity_sandwich}, we see that for some $A,C > 0$ we have
\begin{equation}
\Ru(Y) \le \int\limits_{Y} H\dvol_\sigma \le A\cdot b\le C\cdot\vol{X}^{1/2}\cdot \sys(Y)^{-1/2}
\end{equation}
This implies the desired upper bound, and concludes the proof. \end{proof}



\section{Non-Convex, Dynamically Convex Contact Forms} \label{sec:counter_example} In this section, we use the methods of \cite{abhs2} to construct dynamically convex contact forms with systolic ratio and volume close to $1$, and arbitrarily small and arbitrarily large Ruelle invariant. 

\begin{prop} \label{prop:dyn_cvx_counterexample} For every $\epsilon > 0$, there exists a dynamically convex contact form $\alpha$ on $S^3$ satisfying
\begin{equation}
\label{eqn:dyn_cvx_counterexample_small_ruelle}
\vol{S^3,\alpha} = 1 \qquad \sys(S^3,\alpha) \ge 1 - \epsilon \qquad \Ru(S^3,\alpha) \le \epsilon
\end{equation}
and there exists a dynamically convex contact form $\beta$ on $S^3$ satisfying
\begin{equation}
\label{eqn:dyn_cvx_counterexample_large_ruelle}
\vol{S^3,\beta} = 1 \qquad \sys(S^3,\beta) \ge 1 - \epsilon \qquad \Ru(S^3,\beta) \ge \epsilon^{-1}
\end{equation}
\end{prop}

\subsection{Hamiltonian Disk Maps} \label{subsec:hamiltonian_disk_maps} We begin with some notation and preliminaries on Hamiltonian maps of the disk that we will need for the rest of the section. 

\vspace{3pt}

Let $\D \subset \R^2$ denote the unit disk in the plane with ordinary coordinates $(x,y)$ and polar coordinates $(r,\theta)$. We use $\lambda$ and $\omega$ to denote the standard Liouville form and symplectic form.
\[\lambda := \frac{1}{2}r^2 d\theta = \frac{1}{2}(xdy - ydx) \qquad\text{and}\qquad \omega := r dr \wedge d\theta = dx \wedge dy\]
Let $\phi:[0,1] \times \D \to \D$ be a the Hamiltonian flow (for $t \in [0,1]$) generated by a time-dependent Hamiltonian on $\D$ vanishing on the boundary, i.e.
\[H:\R/\Z \times \D \to \R \qquad\text{with}\qquad H|_{\R/\Z\times\partial \D} = 0\]
We let $X_H$ denote the Hamiltonian vector field and adopt the convention that $\iota_{X_H}\omega = dH$. Since $H$ is constant on the boundary $\partial\D$, the Hamiltonian vector field $X_H$ is tangent to $\partial\D$. Thus $\phi$ is a well-defined flow on $\D$. The differential of $\phi$ defines a map $\Phi:\R \times \D \to \Sp$ with $\Phi|_{0 \times \D} = \on{Id}$, which lifts uniquely to a map
\begin{equation}\label{eqn:lifted_differential_Ham_flow} \widetilde{\Phi}:\R \times \D \to \twSp \qquad\text{satisfying}\qquad \widetilde{\Phi}(S+T,z) = \widetilde{\Phi}(T,\phi_S(z))\widetilde{\Phi}(S,z)\end{equation}

There are two key functions on $\D$ associated to the family of Hamiltonian diffeomorphisms $\phi$. First, there is the action and the associated Calabi invariant.

\begin{definition} The \emph{action} $\sigma_\phi:\D \to \R$ and \emph{Calabi invariant} $\Cal(\D,\phi) \in \R$ of $\phi$ are defined by
\begin{equation}
\sigma_\phi = \int_0^1 \phi^*_t(\iota_{X_H}\lambda + H) \cdot dt \qquad\text{and}\qquad \Cal(\D,\phi) = \int_{\D} \sigma \cdot \omega
\end{equation}
\end{definition}
\noindent The action measures the failure of $\phi$ to preserve $\lambda$, as captured by the following formula.
\begin{equation} \label{eqn:action_calabi_identities}
\phi_1^*\lambda - \lambda = d\sigma_\phi\end{equation}

Next, there is the rotation map and the associated Ruelle invariant. To discuss these quantities, we use the following lemma.

\begin{lemma} \label{lem:rotation_map_on_disk_exists} Let $\phi:[0,1] \times \D \to \D$ be the flow of a Hamiltonian $H:\R/\Z \times \D \to \D$ with $\sigma_\phi > 0$. Then the sequences $r_n:\D \to \R$ and $s_n:\D \to \R$ given by
\[r_n(z) := \frac{1}{n}\rho \circ \widetilde{\Phi}(n,z) \qquad\text{and}\qquad s_n(z) := \frac{1}{n}\sum_{k=0}^{n-1} \sigma_\phi \circ \phi^k(z)\]
converge in $L^1(\D)$ to $r_\phi$ and $s_\phi$, respectively. The map $s_k^{-1}$ also converges to $s_\phi^{-1}$ in $L^1(\D)$.
\end{lemma}

\begin{proof} We apply Kingman's sub-additive ergodic theorem (see \cite{k1973} and Remark \ref{rmk:discrete_kingman}) to the map $g_n = r_n + C$ for sufficiently large $C > 0$. Applying (\ref{eqn:lifted_differential_Ham_flow}) and the quasimorphism property of $\rho$, we find that
\[g_{m + n} \le g_m + g_n \circ \phi^m\]
By Kingman's ergodic theorem, this implies that $\frac{g_n}{n}$ has a limit $r_\infty$ in $L^1(\D)$. Since $\|g_n - r_n\|_{L^1}$ is bounded, we acquire the same result for $r_n$.

\vspace{3pt}

By Birkhoff's ergodic theorem, $s_n$ converges to a limit $s_\infty \in L^1(\D)$. Note that for some $c > 0$, we have
\[c^{-1} \le \sigma_\phi \le c \quad\text{and therefore}\quad c^{-1} \le s_n \le c\]
Thus $s_\infty > 0$ pointwise almost everywhere and $s_\infty^{-1}$ is well-defined almost everywhere. Since $|s_n|^{-1} < c$, we can apply the dominated convergence theorem to conclude that $s_\infty^{-1}$ is integrable and $s_n^{-1} \to s_\infty^{-1}$ in $L^1$. A similar argument applies to $r_n/s_n$, which converges to $r_\infty/s_\infty$.
\end{proof}

\begin{definition} \label{def:calabi_invariant} The \emph{rotation} $r_\phi:\D \to \R$ and \emph{Ruelle invariant} $\Ru(\D,\phi) \in \R$ of $\phi$ are defined by
\begin{equation}
r_\phi := \lim_{n \to \infty} r_n \qquad\text{and}\qquad \Ru(\D,\phi) = \int_\D r_\phi \cdot \omega 
\end{equation}
\end{definition}

\begin{remark} Our Ruelle invariant $\Ru(\D,\phi)$ of a symplectomorphism of the disk agrees with the two-dimensional version of the invariant introduced by Ruelle in \cite{r1985}.
\end{remark}

\noindent The action, rotation, Calabi invariant and Ruelle invariant  depend only on the homotopy class of $\phi$ relative to the endpoints, or equivalently the element in the universal cover of $\on{Ham}(\D,\omega)$.

\vspace{3pt}

We conclude this review with a discussion of periodic points and their invariants.

\begin{definition} A \emph{periodic point} $p$ of $\phi:\D \to \D$ is a point such that $\phi^k(p) = p$ for some $k \ge 1$. The period $\mathcal{L}(p)$, action $\mathcal{A}(p)$ and rotation number $\rho(p)$ of $p$ are given, respectively, by
\begin{equation}
\mathcal{L}(p) := \min{j > 0| \phi^j(p) = p} \qquad \mathcal{A}(p) = \sum_{i=0}^{\mathcal{L}(p)-1} \sigma_\phi \circ \phi^i(p) \qquad \rho(p) := \rho \circ \widetilde{\Phi}(\mathcal{L}(p),p)
\end{equation}
Note that the rotation number can also be written as $\rho(p) = \mathcal{L}(p) \cdot r_\phi(p)$.
\end{definition}

\subsection{Open Books Of Disk Maps} \label{subsec:open_books} We next review the construction of contact forms on $S^3$ from symplectomorphisms of the disk, using open books.  

\begin{construction} \label{con:open_book} Let $H:\R/\Z \times \D \to \R$ be a Hamiltonian with flow $\phi:[0,1] \times \D \to \R$ such that
\begin{itemize}
\item[(i)] Near $\partial \D$, $H$ is of the form $H(t,r,\theta) = B \cdot \pi(1 - r^2)$ for some $B > 0$.
\item[(ii)] The action function $\sigma_\phi$ of the Hamiltonian is positive everywhere.
\end{itemize}
We now construct the \emph{open book} contact form $\alpha$ on $S^3$ associated to $(\D,\phi)$. We proceed by producing two contact manifolds $(U,\alpha)$ and $(V,\beta)$, then gluing them by a strict contactomorphism. 

\vspace{3pt}

To construct $U$, we consider the contact form $dt + \lambda$ on $\R \times \D$. Due to the identity $d\sigma_\phi = \phi^*_1\lambda - \lambda$ in (\ref{eqn:action_calabi_identities}), the map $f$ defined by
\[
f:\R \times \D \to \R \times \D \qquad f(t,z) = (t - \sigma_\phi(z),\phi_1(z))
\]
is a strict contactomorphism. We form the following quotient space.
\[
U = \R \times \D/\sim \qquad\text{defined by } (t,z) \sim f(t,z)
\]
Since $\sigma_\phi$ is strictly positive by assumption (ii) in Construction \ref{con:open_book}, this quotient is a smooth manifold. The contact form $dt + \lambda$ descends to a contact form $\alpha$ on $U$ because $f$ is a strict contactomorphism. Note that a fundamental domain of this quotient is given by
\[\Omega = \{(t,z)|0 \le t \le \sigma_\phi(z)\}\]
We observe that the Reeb vector field is simply given by $R=\partial_t$. Thus the disk $0\times\D\subset U$ is a surface of section for the Reeb flow on $U$ with first return map $\phi_1$.

\vspace{3pt}

To construct $V$, we choose a small $\epsilon > 0$ and let
\[V := \R/\pi\Z \times \D(\epsilon) \qquad \beta := (1 - r^2)dt + \frac{B}{2}r^2d\theta\]
Here $\D(\epsilon) \subset \C$ is the disk of radius $\epsilon$, $t$ is the $\R/\pi \Z$ coordinate and $(r,\theta)$ are radial coordinates on $\D(\epsilon)$. There is a strict contactomorphism $\Psi$ identifying subsets of $U$ and $V$, given by
\[\Psi:V \setminus (\R/\pi\Z \times 0) \to U \qquad\text{with}\qquad \Psi(t,r,\theta) := (\frac{B}{2} \cdot \theta ,\sqrt{1 - r^2}, 2t - B\theta)\]
We now define $Y = \on{int}(U) \cup_\Psi V$ as the gluing of the interior of $U$ and $V$ via $\Phi$, and $\alpha$ as the inherited contact form. Since $\phi$ is Hamiltonian isotopic to the identity, the resulting contact manifold $(Y,\ker\alpha)$ is contactomorphic to standard contact $S^3$.
\end{construction}

\begin{prop}[Open Book] \label{prop:open_book} 
Let $H$ and $\phi$ be as in Construction \ref{con:open_book}. Then there exists a contact form $\alpha$ on $S^3$ with the following properties.
\begin{itemize}
	\item[(a)] (Surface Of Section) There is an embedding $\iota:\D \to S^3$ such that $\iota(\D)$ is a surface of section with return map $\phi_1$ and first return time $\sigma$, and such that $\omega = \iota^*d\alpha$.
	\vspace{3pt}
	\item[(b)] (Volume) The volume of $(S^3,\alpha)$ is given by the Calabi invariant of $(\D,\phi)$, i.e.
	\[\vol{S^3,\alpha} = \Cal(\D,\phi)\]
	\item[(c)] (Ruelle) The Ruelle invariant of $(S^3,\alpha)$ is given by a shift of the Ruelle invariant of $(\D,\phi)$.
	\[\Ru(S^3,\alpha) = \Ru(\D,\phi) + \pi\]
	\item[(d)] (Binding) The binding $b = \iota(\partial \D)$ is a Reeb orbit of action $\pi$ and rotation number $1 + 1/B$. 
	\vspace{3pt}
	\item[(e)] (Orbits) Every simple orbit $\gamma \subset S^3 \setminus b$ corresponds to a periodic point $p$ of $\phi$ that satisfies
	\[\on{lk}(\gamma,b) = \mathcal{L}(p) \qquad \mathcal{A}(\gamma) = \mathcal{A}(p) \qquad \rho(\gamma) = \rho(p) + \mathcal{L}(p)\]
\end{itemize}
\end{prop}

In order to relate various invariants associated to $(S^3,\alpha)$ and its Reeb orbits to corresponding structures for $(\D,\phi)$, we need to introduce a certain trivialization of $\xi$ over $U$.

\begin{construction} \label{con:open_book_trivialization} Let $(U,\xi|_U)$ be as in Construction \ref{con:open_book}. We let $\tau$ denote the continuous trivialization of $\xi|_U$ defined as follows. On the fundamental domain $\Omega$, we let
\begin{equation}\label{eqn:open_book_trivialization}
\tau:\Omega \to \on{Hom}(\xi|_U,\R^2) \qquad\text{given by}\qquad
\tau(t,z) := \exp(2\pi i t/\sigma_\phi(z)) \circ \Phi(t/\sigma_\phi(z),z) \circ \Pi_\D
\end{equation}
Here $\Phi:[0,1] \times \D \to \D$ is the differential $d\phi$ of the flow $\phi:[0,1] \times \D \to \D$ and $\Pi_\D:\xi \to T\D$ denotes projection to the (canonically trivial) tangent bundle $T\D$ of $\D$. Note also that $\circ$ denotes composition of bundle maps.

\vspace{3pt}

To check that $\tau$ descends to a well-defined trivialization on $U$, we must check that it is compatible with the quotient map $f:\R \times \D \to \R \times \D$. Indeed, we have
\[\tau(\sigma_\phi(z),z) = \Phi(1,z) \circ \Pi_\D = \tau(0,\phi_1(z)) \circ df_{(\sigma_\phi(z),z)}\]
This precisely states that projection commutes with the isomorphism identifying tangent spaces in the quotient, so $\tau$ descends from $\Omega$ to $U$.
\end{construction}

\begin{lemma} \label{lem:open_book_trivializations} Let $\tau:\xi|_U \to \R^2$ be the trivialization in Construction \ref{con:open_book_trivialization}. Then
\begin{itemize}
	\item[(a)] The restriction $\tau|_K$ of $\tau$ to any compact subset $K \subset \on{int}(U)$ of the interior of $U$ is the restriction of a global trivialization of $\xi$ on $S^3$.
	\vspace{3pt}
	\item[(b)] The local rotation number $\rot_\tau:U \to \R$ of $(U,\alpha|_U)$ with respect to $\tau$ agrees with the restriction of the local rotation number $\rot:S^3 \to \R$ of $(S^3,\alpha)$ with respect to the global trivialization.
\end{itemize}
\end{lemma}

\begin{proof} Let $V = \R/\pi\Z \times \D(\epsilon)$ and $\Psi$ be as in Construction \ref{con:open_book}. For any $\delta < \epsilon$, we let $V(\delta) \subset V$ and $U(\delta) \subset U$ denote
\[V(\delta) := \R/\pi\Z \times \D(\delta) \subset V \qquad\text{and}\qquad U(\delta) := \on{int}(U) \setminus \on{int}(\Psi(V(\delta)))\]
The sets $U(\delta)$ are an exhaustion of $\on{int}(U)$ by compact, Reeb-invariant contact submanifolds.

\vspace{3pt}

To show (a), we assume that $K = U(\delta)$. The homotopy classes of trivializations $\mathcal{T}$ of $\xi$ over $U(\delta)$ are in bijection with $H^1(U(\delta);\Z) \simeq \Z$. A map to $\Z$ classifying elements of $\mathcal{T}$ is given by
\begin{equation}
\label{eqn:classifying_map_for_trivializations}
\mathcal{T} \to \Z \qquad\text{given by}\qquad \sigma \mapsto \on{sl}(\gamma,\sigma)
\end{equation}
Here $\on{sl}(\gamma,\sigma)$ is the self-linking number (in the trivialization $\sigma$) of the following transverse knot.
\[\gamma:\R/2\pi\Z \to U(\delta) \qquad \gamma(\theta) = \Psi(0,\epsilon,\theta) = (\frac{B\theta}{2}, \sqrt{1 - \epsilon^2}, - B\theta)\]
The knot $\gamma$ bounds a Seifert disk $\Sigma = 0 \times \D(\epsilon)$ in $V \subset S^3$. The line field $\xi \cap \Sigma$ has a single positive elliptic singularity, so the self-linking number of the boundary $\gamma$ with respect to the global trivialization is $\on{sl}(\gamma) = -1$.

\vspace{3pt}

To compute $\on{sl}(\gamma,\tau)$, we push $\gamma$ into $\Sigma$ along a collar neighborhood to acquire a nowhere zero section $\eta:\R/2\pi\Z \to \xi$ and then compose with $\tau$ to acquire a map $\tau \circ \eta:\R/2\pi\Z \to \R^2 \setminus 0$. Up to isotopy through nowhere zero sections, we can compute that
\[\tau \circ \eta(\theta) = e^{i\theta} \in \C = \R^2 \]
On the other hand, the self-linking number can be computed as the negative of the winding number of this map.
\[\on{sl}(\gamma,\tau) = -\on{wind}(\tau \circ \eta) = -1\]
Since the map \eqref{eqn:classifying_map_for_trivializations} classifies elements of $\mathcal{T}$, this proves that on $U(\delta)$ the trivialization $\tau$ agrees with the restriction of a global trivialization.

\vspace{3pt}

To show (b), note that since $U(\delta)$ is compact, we can choose a global trivialization of $\xi$ on $S^3$
\[\sigma:\xi \simeq \R^2 \qquad \text{such that}\qquad \sigma|_{U(\delta)} = \tau|_{U(\delta)}\]
By Proposition \ref{prop:general_ruelle_invariant}(c), $\rot_\sigma = \rot$ on $S^3$ and so  the local rotation numbers satisfy
\[\rot|_{U(\delta)} = \rot_\sigma|_{U(\delta)} = \rot_\tau|_{U(\delta)}\]
Since this holds for any $\delta$, this shows (b) on all of $\on{int}(U)$. 
\end{proof}

\begin{remark} It is possible to define the trivialization $\tau$ of $\xi|_U$ in such a way that it is not only homotopic but actually equal to the restriction of a global trivialization of the contact structure on $S^3$. We did not do this in order to avoid complicated formulas in the definition of $\tau$. \end{remark}

The following lemma relates the rotation number of the lifted linearized Reeb flow $\widetilde{\Phi}_\tau:\R\times U\rightarrow \twSp$ on $(U,\alpha)$ to the rotation number of the lifted linearized flow $\widetilde{\Phi}:\R\times\D\rightarrow\twSp$ of the Hamiltonian flow $\phi$ on $\D$.

\begin{lemma}
\label{lem:open_book_trivializations_rotation_number}
Let $\iota:\D\rightarrow U$ denote the inclusion of the disk $0\times\D\subset U$. Let $p\in\D$ and consider the Reeb trajectory $\gamma:\R\rightarrow U$ satisfying $\gamma(0)=\iota(p)$. Let $0=T_0 < T_1 < T_2 < \dots$ denote the non-negative times at which the trajectory $\gamma$ intersects the disk $\iota(\D)$. Then
\begin{equation*}
\rho\circ \widetilde{\Phi}_\tau(T_k,\iota(p)) = \rho\circ \widetilde{\Phi}(k,p) + k
\end{equation*}
for all non-negative integers $k$.
\end{lemma}

\begin{proof}
We abbreviate
\[p_i=\phi^i(p) \qquad y_i = \iota(p_i) = \gamma(T_i) \qquad L_i = T_{i+1} - T_i = \sigma_\phi(p_i)\]
Note that the lifted linearized Reeb flow with respect to $\tau$ at time $T_k$ can be written as
\begin{equation} \label{eqn:open_book_orbit_rotation_a}
\widetilde{\Phi}_\tau(T_k,y_0) = \widetilde{\Phi}_\tau(L_{k-1},y_{k-1}) \widetilde{\Phi}_\tau(L_{k-2},y_{k-2}) \dots \widetilde{\Phi}_\tau(L_0,y_0)
\end{equation} 
The linearized Reeb flow $\widetilde{\Phi}_\tau(L_i,y_i)$ takes place along a trajectory connecting $(0,p_i)$ to $(\sigma_\phi(p_i),p_i)$ in the fundamental domain $\Omega$. We may be directly compute from (\ref{eqn:open_book_trivialization}) that
\begin{equation} \label{eqn:open_book_orbit_rotation_b}
\Phi_\tau(t,y_i) = \exp(2\pi i t/\sigma_\phi(p_i)) \circ \Phi(t/\sigma_\phi(z),p_i) \quad\text{and so}\quad \widetilde{\Phi}_\tau(L_i,y_i) = \widetilde{\Xi} \cdot \widetilde{\Phi}(1,p_i)
\end{equation}
Here $\widetilde{\Xi}$ is the unique lift of $\on{Id} \in \Sp$ with $\rho(\widetilde{\Xi}) = 1$. This is a central element of $\twSp$, so combining (\ref{eqn:open_book_orbit_rotation_a}) and (\ref{eqn:open_book_orbit_rotation_b}) we have
\[\widetilde{\Phi}_\tau(T_k,y_0) = \widetilde{\Xi}^k \cdot \widetilde{\Phi}(1,\phi^{k-1}(p)) \cdot \widetilde{\Phi}(1,\phi^{k-2}(p)) \cdots \widetilde{\Phi}(1,p) = \widetilde{\Xi}^k \cdot \widetilde{\Phi}(k,p)\]
Since $\rho(\widetilde{\Xi} \cdot \widetilde{\Psi}) = 1 + \rho(\widetilde{\Psi})$ for any $\widetilde{\Psi} \in \twSp$, we can conclude that
\[\rho\circ \widetilde{\Phi}_\tau(T_k,\iota(p)) = \rho\circ \widetilde{\Phi}(k,p) + k \qedhere\]
\end{proof}

\begin{proof}[Proof of Proposition \ref{prop:open_book}]
We prove each of the properties (a)-(e) separately.

\vspace{3pt}

{\bf Surface Of Section.} Define the inclusion $\iota:\D \to S^3$ as the following composition.
\[\iota:\D = 0 \times \D \to \R \times \D \xrightarrow{\pi} Y \simeq S^3\]
The surface $0 \times \D$ is transverse to the Reeb vector field $\partial_t$ of $\R \times \D$ and intersects every flowline $\R \times z$. Also, $(\R \times z) \cap \Omega$ has action $\sigma_\phi(z)$ and ends on $(\sigma_\phi(z),z) \sim (0,\phi_1(z))$. Thus $\iota(\mathbb{D}) = \pi(0 \times \D)$ is a surface of section with return time $\sigma_\phi$ and monodromy $\phi_1$. Finally, note that
\[\iota^*(d\alpha) = d(dt + \lambda)|_{0 \times \D} = \omega\]
This verifies all of the properties of $\iota:\D \to Y\simeq S^3$ listed in (a). 

\vspace{3pt}

{\bf Calabi Invariant.} This property follows from a simple calculation of the volume using the fundamental domain $\Omega$.
\[\vol{Y,\alpha} = \int_Y \alpha \wedge d\alpha = \int_\Omega dt \wedge d\lambda = \int_\D \sigma_\phi \cdot \omega = \Cal(\D,\phi)\]

\vspace{3pt}

{\bf Ruelle Invariant.} Let $\rot:S^3 \to \R$ be the local rotation number of $(S^3,\alpha)$. By Lemma \ref{lem:open_book_trivializations}, the restriction of $\rot$ to the (open) fundamental domain $\Omega \subset S^3$ coincides with $\rot_\tau$. Since $S^3 \setminus \Omega$ is measure $0$ in $S^3$, we thus have 
\begin{equation} \label{eqn:open_book_ruelle_a}
\Ru(S^3,\alpha) = \int_{S^3} \rot \cdot \alpha \wedge d\alpha = \int_\Omega \rot_\tau \cdot dt \wedge \omega = \int_\D \iota^*\rot_\tau \cdot \sigma_\phi \omega\end{equation}
Here $\iota^*\rot_\tau$ denotes the pullback of $\rot_\tau$ via the map $\iota:\D \to S^3$ from (a). We have used the fact that $\rot_\tau$ is constant along Reeb trajectories. This follows directly from the definition of $\rot_\tau$ as a time average.

\vspace{3pt}

To apply this alternative formula for $\Ru(S^3,\alpha)$, let $T_k$ denote the $k$th positive time that the Reeb trajectory $\gamma:[0,\infty) \to S^3$ intersects the surface of section $\iota(\D)$. Then 
\[
\iota^*\rot_\tau = \lim_{k \to \infty} \frac{\rho \circ \widetilde{\Phi}_\tau(T_k,-)}{T_k} = \lim_{k \to \infty} \frac{\rho \circ \widetilde{\Phi}(k,-) + k}{\sum_{i=0}^{k-1} \sigma_\phi \circ \phi^i} = \frac{r_\phi + 1}{s_\phi}
\]
Here the second equality is a consequence of Lemma \ref{lem:open_book_trivializations_rotation_number}. The maps $r_\phi$ and $s_\phi$ are the averaged rotation and action maps constructed in Lemma \ref{lem:rotation_map_on_disk_exists}. By construction, these maps are invariant under pullback by $\phi$. Thus 
\[
\int_\D \frac{r_\phi + 1}{s_\phi} \cdot \sigma_\phi\omega = \frac{1}{n}\sum_{k=0}^{n-1} \int_\D [\phi^k]^*(\frac{r_\phi + 1}{s_\phi} \cdot \sigma_\phi \omega) = \int_\D \frac{r_\phi + 1}{s_\phi} \cdot s_n \omega \quad\text{where}\quad s_n = \frac{1}{n} \sum_{k=0}^{n-1} \sigma_\phi \circ \phi^k
\]
By Lemma \ref{lem:rotation_map_on_disk_exists}, we know that $s_n \to s_\phi$ in $L^1(\D)$. Thus, by combining the above formula in the $n \to \infty$ limit with (\ref{eqn:open_book_ruelle_a}), we acquire the desired property.
\[
\Ru(S^3,\alpha) = \int_\D \frac{r_\phi + 1}{s_\phi} \cdot \sigma_\phi \cdot \omega = \int_\D \frac{r_\phi + 1}{s_\phi} \cdot s_\phi \cdot \omega = \int_\D (r_\phi + 1)\cdot \omega = \Ru(\D,\phi) + \pi
\]
\vspace{3pt}

{\bf Binding.} Let $b = \iota(\partial \D)$ be the binding which coincides with $\R/\pi\Z \times 0$ in $V$. First note that the Reeb vector field is given on $(V,\beta)$ by the following formula.
\begin{equation} \label{eqn:Reeb_vectorfield_on_V} 
R_\beta = \partial_t + \frac{2}{B} \partial_\theta\end{equation}
Thus $b$ is a Reeb orbit. Since $b$ bounds a symplectic disk $\iota(\D) \subset S^3$ of area $\pi$, the action is $\pi$. To compute $\rho(b)$, note that there is a natural trivialization of $\xi|_V = \ker(\beta)$ given by
\[
\nu:\xi|_V \subset TV \xrightarrow{\pi} T\D(\epsilon) = \R^2
\]
The Reeb flow $\phi:\R \times V \to V$ and the linearized Reeb flow $\Phi_\nu:\R \times V \to \Sp$ with respect to $\nu$ can be calculated from (\ref{eqn:Reeb_vectorfield_on_V}), as follows. 
\[\phi_t(s,z) = (s + t, e^{2it/B} \cdot z) \qquad \Phi_\nu(t,s,z) = e^{2it/B}\]
Thus the rotation number $\rho(b,\nu)$ of $b$ in the trivialization $\nu$ is $1/B$. Finally, to compute the rotation number $\rho(b) = \rho(b,\tau)$ with respect to the global trivialization $\tau$ on $\xi$, we note that
\[\rho(b,\tau) - \rho(b,\nu) = \mu(\tau \circ \nu^{-1}|_b) = c_1(\xi|_{\iota(\D)},\tau) - c_1(\xi|_{\iota(\D)},\nu) = - c_1(\xi|_{\iota(\D)},\nu)\]
Here $\mu:\pi_1(\Sp) \to \Z$ is the Maslov index and $c_1(\xi|_{\iota(\D)},-)$ is the relative Chern class of $\xi|_{\iota(\D)}$ with respect to a given trivialization over $\iota(\partial \D)$, which vanishes for $\tau$. 

\vspace{3pt}

On the other hand, the trivialization $\nu$ is specified by the section of $\xi|_{\iota(\D)}$ given by pushing $\iota(\partial \D)$ into $\iota(\D)$ along a collar neighborhood. Thus, $-c_1(\xi|_{\iota(\D)},\nu)$ is precisely the self-linking number $\on{sl}(b)$ of $b$. This number can be calculated as a signed count of singularities of the line field $\xi \cap \iota(\D)$, which has $1$ elliptic singularity. Thus $\on{sl}(b) = -1$ and $\rho(b) = 1 + 1/B$.

\vspace{3pt}

{\bf Orbits.} An embedded closed orbit $\gamma:\R/L\Z\rightarrow Y$ of $\alpha$ that is disjoint from the binding $b$ is equivalent to a closed orbit of $(U,\alpha|_U)$. The orbit $\gamma$ intersects the surface of section $\iota(\D)$ transversely at $n \ge 1$ times $T_0 = 0,T_1,\dots,T_n = L$. Let
\[p_k \in \D \qquad\text{be such that}\qquad \iota(p_k) = \gamma(T_k) \cap \iota(\D)\]
Since $\iota(\D)$ is a surface of section, we have $p_{i+1} = \phi(p_i)$ and since $\gamma$ is closed, $p_n = p_0$. Thus $p = p_0$ is a periodic point of period
\[\mathcal{L}(p) = n = \iota_*[\D] \cdot [\gamma] = \on{lk}(\gamma,b)\]
Next, note that on the interval $[T_i,T_{i+1}]$, $\gamma$ restricts to a map $[T_i,T_{i+1}] \to \Omega$ given by $\gamma(t) = (t,\iota(p_i))$, from which it follows that
\[
\mathcal{A}(\gamma) = \sum_{k=0}^{n-1} \int_{T_k}^{T_{k+1}} \gamma^*(dt + \alpha) = \sum_{k=0}^{n-1} \int_0^{\sigma(p_k)} dt = \sum_{k=0}^{n-1} \sigma \circ \phi^k(p) = \mathcal{A}(p)
\]
Finally, due to Lemma \ref{lem:open_book_trivializations} we may use the trivialization $\tau$ to compute the rotation number. We have
\[\rho(\gamma) = \rho \circ \widetilde{\Phi}_\tau(L,\gamma(0)) = \rho \circ \widetilde{\Phi}(n,p) + n = \rho(p) + \mathcal{L}(p)\]
Here the second equality uses Lemma \ref{lem:open_book_trivializations_rotation_number}. This completes the proof of (e), and the entire proposition.
\end{proof}

\subsection{Radial Hamiltonians} \label{subsec:radial_Hamiltonians} A Hamiltonian $H:\R/\Z \times \D \to \R$ that is rotationally invariant will be called \emph{radial}. In other words, $H$ is radial if it can be written as
\[H(t,r,\theta) = h(t,r) \qquad\text{for a map}\qquad h:\R/\Z \times [0,1] \to \R\]
We will require a few lemmas regarding radial Hamiltonians.
\begin{lemma} \label{lem:radial_action_function_1} Let $H:\D \to \R$ be an autonomous, radial Hamiltonian with $H = h \circ r$. Then
\begin{equation}
\sigma_\phi(r,\theta) = h(r) - \frac{1}{2}rh'(r) \qquad\text{and}\qquad r_\phi(r,\theta) = -\frac{h'(r)}{2\pi r}
\end{equation}
\end{lemma}

\begin{proof} We calculate the Hamiltonian vector field $X_H$ and the action function $\sigma_\phi$ as follows.
\[X_H = -\frac{h'}{r} \cdot \partial_\theta \qquad\text{and}\qquad\text{and}\qquad \sigma_\phi(r,\theta) = \int_0^1 \phi^*_t(-\frac{rh'(r)}{2} + h(r)) \cdot dt = h(r) - \frac{1}{2}rh'(r)\]
Here we use the fact that the Hamiltonian flow $\phi$ preserves any function of $r$. Next, we note that the differential $\Phi:\R \times \D \to \D$ of the flow $\phi$ is given by
\[\Phi(t,z)v = \exp(\frac{-h'}{r} \cdot it)v + \frac{it(rh'' - h')}{r^2} \cdot \exp(\frac{-h'}{r} \cdot it)z \cdot dr(v)\]
Note that if we use $s = iz/|z|$, then $dr(v) = 0$. Thus, if $\widetilde{\Phi}:\R \times \D \to \twSp$ denotes the lift of $\Phi$, and $\rho_s$ denotes the rotation number relative to $s$ (see Definition \ref{def:rotation_number_rel_s}) then
\begin{equation}\label{eqn:radial_action_function_1a}
\Phi(t,z)s = \exp(\frac{-h'(r)}{r} \cdot it)s \qquad\text{and thus}\qquad \rho_s \circ \widetilde{\Phi}(T,z) = T \cdot \frac{-h'(r)}{2\pi r}\end{equation}
Since $\rho_s:\twSp \to \R$ and $\rho:\twSp \to \R$ are equivalent quasimorphisms (Lemma \ref{lem:rhos_all_equivalent}), we have

\[r_\phi = \lim_{T \to \infty} \frac{\rho \circ \widetilde{\Phi}(T,-)}{T} = \lim_{T \to \infty} \frac{\rho_s \circ \widetilde{\Phi}(T,-)}{T} = \frac{-h' \circ r}{2\pi r} \qquad\text{in }L^1(\D)\]
This concludes the proof of the lemma. \end{proof}

More generally, a Hamiltonian $H:\R/\Z \times \D \to \R$ is called \emph{radial around} $p \in \D$ if $H$ is invariant under rotation around $p$, i.e. if $H$ can be written as
\[H(t,x,y) = h(t,r_p) \qquad\text{for a map}\qquad h:\R/\Z \times [0,1] \to \R\]
Here $r_p:\D \to \R$ be the distance from $p$, i.e. $r_p(z) = |z - p|$. 

\begin{lemma} \label{lem:radial_action_function_2} Let $H:\D \to \R$ be an autonomous Hamiltonian that is radial around $p = (a,b) \in \D$, with $H = h \circ r_p$, in a neighborhood $U$ of $p$. Then on $U$, we have
\begin{equation}\label{eqn:radal_action_function_2}\sigma_\phi = h(r_p) - \frac{1}{2}r_ph'(r_p) + u_p - \phi_1^*u_p \qquad \text{and}\qquad r_\phi = -\frac{h'(r_p)}{2\pi r_p}\end{equation}
Here the map $u_p:\D \to \R$ is given by $u_p(x,y) = (bx - ay)/2$.
\end{lemma}

\begin{proof} Let $\lambda_p$ be the radial Liouville form on $(\D,\omega)$ centered at $p$. That is, $\lambda_p$ is given by
\[\lambda_p = \frac{1}{2}((x - a)dy - (y - b)dx) = \lambda + du_p\]
Let $\tau:\D \to \R$ be the function decribed in (\ref{eqn:radal_action_function_2}). Then by Lemma \ref{lem:radial_action_function_1}, we know that on $U$ we have
\[
d\tau = (\phi_1^*\lambda_p - \lambda_p) + (\phi_1^*du_p - u_p) = \phi^*_1\lambda - \lambda = d\sigma_\phi\]
Thus it suffices to check that $\sigma_\phi(p) = \tau(p)$. Since $rh'(p) = 0$ and $u_p(p) = u_p(\phi_1(p)) = 0$, we see that $\tau(p) = h(0) = H(p)$. On the other hand, $X_H(p) = 0$, we see that
\[\sigma_\phi(p) = \int_0^1 \phi^*_t(\lambda(X_H) + H)dt = \int_0^1 h(0)dt = \tau(p)\]
Thus $\sigma_\phi(p) = \tau(p)$. The formula for $r_\phi$ follows from identical arguments to Lemma \ref{lem:radial_action_function_1}.\end{proof}

\subsection{Special Hamiltonian Maps} \label{subsec:special_hamiltonian_map} We next construct a special Hamiltonian flow $\phi$ on the disk, depending on a set of parameters, and establish its basic properties with respect to action, rotation and periodic orbits. The desired contact forms in Proposition \ref{prop:dyn_cvx_counterexample} with small and large Ruelle invariant correspond (via Proposition \ref{prop:open_book}) to $\phi$ for suitable choices of parameters (see \S \ref{subsec:main_construction}).

\vspace{3pt}

The special Hamiltonian flow $\phi$ is constructed as the product of a pair of simpler flows.
\[
\phi = \phi^H \bullet \phi^G
\]
Here $\phi^G:[0,1] \times \D \to \D$ and $\phi^H:[0,1] \times \D \to \D$ are autonomous flows generated by $G$ and $H$, and the product $\bullet$ occurs in the universal cover of the group $\on{Ham}(\D,\omega)$ of Hamiltonian diffeomorphisms of $(\D,\omega)$. We denote the Hamiltonian generating $\phi$ by
\[H \# G:\R/\Z \times \D  \to \R\]

\begin{setup} \label{set:special_hamiltonian} We will require the following setup in the construction of $\phi$. The setup and notation established here will be used for the remainder of of \S \ref{subsec:special_hamiltonian_map}.
\begin{itemize}
\item[(a)] Fix an integer $n \ge 10$ and let $\mathbb{S}(n,k) \subset \D$ for $0 \le k \le n-1$ be the sector of points 
\[\mathbb{S}(n,k) := \{re^{i\theta} \in \D \; : \; 2\pi k/n < \theta < 2\pi(k+1)/n\}\]
\item[(b)] Fix a finite union $U \subset \D$ of disjoint disks in $\D$ such that each of the component disks $D\subset U$ is contained in one of the sectors $\mathbb{S}(n,k)$ and such that for every $D\subset U$ the disk $e^{2\pi i/n} \cdot D$ is a component disk of $U$ as well. We let
\[d(U) := \max{\on{diam}(D) \; : \; D \subset U \text{ is a component disk}}\]
That is, $d(U)$ is the maximal diameter of a disk in $U$.
\item[(c)] Fix a constant $\delta > 0$ that is much smaller than the radius of each disk $D$, than the distance between any two of the disks $D$ and $D'$, and than the distance between $D$ and the boundary of any of the sectors $\mathbb{S}(n,k)$. For any subset $S \subset \D$, we use the notation
\[N(S) := \{z \in \D | |z - p| \le \delta \text{ for some }p \in S\}\]
The neighborhoods $N(\partial D)$, $N(D)$, $N(U)$ and $N(\partial U)$ will be of particular importance.
\item[(d)] Fix a real number $R\in\R$. For every $s>\delta$ sufficiently large compared to $\delta$, there exists a smooth, monotonic function $g_s:[0,s+\delta]\to\R$ with support contained in $[0,s+\delta)$ satisfying the following conditions.
\begin{equation}
\label{eqn:g_s_properties_a}
g_s(r) = \pi\cdot R\cdot (s^2 - r^2) \qquad \text{if} \enspace r\leq s-\delta
\end{equation}
\begin{equation}
\label{eqn:g_s_properties_b}
|g_s'(r)| \leq 2\pi\cdot |R| \cdot (s-\delta)\qquad \text{if} \enspace s-\delta \leq r\leq s+\delta 
\end{equation}
Choose such a function $g_s$ for every $s$ that arises as the radius of a component disk $D\subset U$.
\end{itemize}
\end{setup}

We now introduce the two Hamiltonians $H$ and $G$ in some detail. The construction of $H$ only depends on the integer $n$. The construction of $G$ depends on $U$, $\delta$, $R$ and the choice of $g_s$.

\begin{construction} We let $H:\D \to \R$ denote the radial Hamiltonian given by the formula
\begin{equation} \label{eqn:def_of_H}
H(r,\theta) := \frac{\pi(n+1)}{n} \cdot (1 - r^2)
\end{equation}
The Hamiltonian vector field is $X_H = \frac{2\pi(n+1)}{n} \cdot \partial_\theta$ and so the Hamiltonian flow is given by
\begin{equation}\phi^H:\R \times \D \to \D \qquad\text{with}\qquad \phi^H(t,z) = \exp(\frac{2\pi(n+1)}{n} \cdot it) \cdot z\end{equation}
In particular, the time $1$ flow is rotation by $\frac{2\pi}{n}$ and preserves the collection $U$. \end{construction} 

\begin{construction} We let $G:\D \to \R$ denote a Hamiltonian that is invariant under rotation by angle $2\pi/n$ and that vanishes away from $N(U)$. That is
\begin{equation} \label{eqn:Hamiltonian_G_a}
G(z) = G(e^{2\pi i/n} \cdot z) \qquad\text{and}\qquad G|_{\D \setminus N(U)} = 0\end{equation}
Furthermore, let $D$ be a component disk of $U$ that is centered at $p \in \D$ and with radius $s$. Then we assume that $G$ is given by
\begin{equation}
G|_{N(D)} = g_s\circ r_p
\end{equation}
in the neighbourhood $N(D)$ of the disk $D$. This fully specifies $G$ on all of $\D$.
\end{construction}

A crucial fact that we will use later without comment is that $\phi^G$ and $\phi^H$ commute as elements of the universal cover of $\on{Ham}(\D,\omega)$, i.e. $\phi^G \bullet \phi^H = \phi^H \bullet \phi^G$. The remainder of this section is devoted to calculating properties of the action, rotation and periodic points of the map $\phi$.

\begin{lemma}[Action of $\phi$] \label{lem:action_of_phi} The action map $\sigma_\phi:\D \to \R$ and Calabi invariant $\Cal(\D,\phi)$ satisfy
\begin{equation} \label{eqn:special_hamiltonian_action_a}
\sigma_\phi = \pi(1 + \frac{1}{n}) + R\sum_{D \subset U} \area(D) \cdot \chi_D + O(d(U)) \qquad \text{on} \quad \D \setminus N(\partial U)\end{equation}
\begin{equation} \label{eqn:special_hamiltonian_action_b}
\pi/2 + \min{0,R}\cdot 2\pi/n \le \sigma_\phi \le 2\pi + \max{0,R}\cdot 2\pi/n \qquad \text{on}\quad \D
\end{equation}
\begin{equation} \label{eqn:special_hamiltonian_action_c}
\Cal(\D,\phi) = \pi^2(1 + \frac{1}{n}) +R\sum_{D \subset U } \area(D)^2 + O(d(U)) + O(|R|\cdot \area(N(\partial U)))\end{equation}

\end{lemma}

\begin{proof} Since $\phi^G$ and $\phi^H$ commute, we have $\sigma_G \circ \phi^H_1 = \sigma_G$ and therefore
\[\sigma_\phi = \sigma_G \circ \phi^H_1 + \sigma_H = \sigma_G + \sigma_H\]
Thus we must compute the action map of $G$ and $H$. First, we note that $H$ is radial by (\ref{eqn:def_of_H}). Thus we apply  Lemma \ref{lem:radial_action_function_1} to see
\begin{equation} \label{eqn:special_hamiltonian_action_proof_a} \sigma_H = \pi(1 + \frac{1}{n}) \quad \text{on all of } \D\end{equation}
Next we compute the action map of $G$. Let $D$ be a component disk of $U$ centered at $p$ and of radius $s$. We can apply Lemma \ref{lem:radial_action_function_2} to see that 
\[\sigma_G = g_s(r_p)- \frac{1}{2} r_p g_s'(r_p) + (u_p - (\phi^G_1)^*u_p) = R \area(D) + O(d(U)) \quad \text{on } D \setminus N(\partial D)\]
Here we used expression \eqref{eqn:g_s_properties_a} for $g_s$. It follows from the definition of $u_p$ in Lemma \ref{lem:radial_action_function_2} and the fact that $d(U)$ is an upper bound on the diameter of $D$ that $|u_p - (\phi^G_1)^*u_p|$ is bounded above by $d(U)$. Since $\sigma_G = 0$ outside of $N(D)$, we thus acquire the formula
\begin{equation} \label{eqn:special_hamiltonian_action_proof_b} 
\sigma_G = R\sum_{D \subset U} \area(D) \cdot \chi_D +O(d(U)) \quad \text{ on }\D \setminus N(\partial U) \end{equation}
Adding (\ref{eqn:special_hamiltonian_action_proof_a}) and (\ref{eqn:special_hamiltonian_action_proof_b}) yields the desired formula (\ref{eqn:special_hamiltonian_action_a}) and implies (\ref{eqn:special_hamiltonian_action_b}) away from $N(\partial U)$. We can estimate on the neighbourhood $N(\partial D)$
\[|\sigma_G| \le |g_s(r_p) - \frac{1}{2}r_p g_s'(r_p)| + |u_p - (\phi^G_1)^*u_p| \le 2|R|\pi s^2 + |u_p - (\phi^G_1)^*u_p|\]
We observe that $\pi s^2 < \pi/n$ and $|u_p - (\phi^G_1)^*u_p|< 1$. Moreover, $\sigma_G\ge 0$ if $R\ge 0$ and $\sigma_G\le 0$ if $R\le 0$. Adding $\sigma_G$ to the formula (\ref{eqn:special_hamiltonian_action_proof_a}) for $\sigma_H$ thus yields (\ref{eqn:special_hamiltonian_action_b}) on $N(\partial D)$. Finally, since $|\sigma_G|=O(|R|)$, the integral of $\sigma_G$ over $\D$ agrees with the integral over $\D\setminus N(\partial U)$ up to an $O(|R|\cdot\area(N(\partial U)))$ term. This proves (\ref{eqn:special_hamiltonian_action_c}).
\end{proof}

\begin{lemma}[Rotation of $\phi$] \label{lem:rotation_of_phi} The rotation map $r_\phi:\D \to \R$ and the Ruelle invariant $\Ru(\D,\phi)$ satisfy
\begin{equation} \label{eqn:special_hamiltonian_rotation_a}
r_\phi = (1 + \frac{1}{n}) + R\sum_{D \subset U} \chi_D \qquad \text{on} \quad \D \setminus N(\partial U)\end{equation}
\begin{equation} \label{eqn:special_hamiltonian_rotation_b}
1+\frac{1}{n} + \min{0,R} \leq r_\phi  \leq 1+\frac{1}{n} + \max{0,R} \qquad \text{on}\quad \D
\end{equation}
\begin{equation} \label{eqn:special_hamiltonian_rotation_c}
\Ru(\D,\phi) = \pi(1 + \frac{1}{n}) + R \sum_{D \subset U} \area(D) + O(R\cdot \area(N(\partial U)))\end{equation}
\end{lemma}

\begin{proof} In the universal cover of $\on{Ham}(\D,\phi)$, the time $k$ flow $\phi^k$ of $G \# H$ can be factored in terms of the time $1$ flow $\phi^G:[0,1] \times \D \to \D$ of $G$ and the time $1$ flow $\phi^H:[0,1] \times \D \to \D$ of $H$, as follows.
\[\phi^k = (\phi^H \bullet \phi^G)^k = \phi^H \bullet \phi^G \bullet \phi^H \bullet \dots \bullet \phi^H \bullet \phi^G\]
This factorization is inherited by the lifted differential $\tilde{\Phi}:\R \times \D \to \twSp$ of $\phi:\R \times \D \to \D$ due to the cocycle property of $\widetilde{\Phi}$.
\begin{equation}\label{eqn:special_hamiltonian_lifted_differential}
\widetilde{\Phi}(k,z) = \widetilde{\Phi}^H(1,\phi^G \circ \phi^{k-1}(z)) \bullet \widetilde{\Phi}^G(1,\phi^{k-1}(z)) \bullet \widetilde{\Phi}^H(1,\phi^G \circ \phi^{k-2}(z)) \bullet \dots \bullet \widetilde{\Phi}^G(1,z)\end{equation}

To apply this, we note that the differential $\Phi^H:[0,1] \times \D \to \Sp$ of the flow of $H$ is given by
\begin{equation}\label{eqn:special_hamiltonian_lifted_differential_a}
\Phi^H(t,z) = \exp(2\pi(1+1/n) \cdot it) \quad\text{for any }z \in \D \end{equation}
Likewise, the differential $\Phi^G:[0,1] \times \D \to \Sp$ of the flow of $G$ is given by the formula
\begin{equation}\label{eqn:special_hamiltonian_lifted_differential_b}
\Phi^G(t,z) = \exp(2\pi \cdot R \cdot it)\text{ if }z \in U\setminus N(\partial U) \qquad\text{and}\qquad \Phi^G(t,z) = \on{Id} \text{ if }z \in \D \setminus N(D)\end{equation}
By combining (\ref{eqn:special_hamiltonian_lifted_differential_a}) and (\ref{eqn:special_hamiltonian_lifted_differential_b}) with the decomposition (\ref{eqn:special_hamiltonian_lifted_differential}), we acquire the following formula.
\begin{equation} \label{eqn:special_hamiltonian_lifted_differential_c}
\rho \circ \widetilde{\Phi}(k,z) = k \cdot \big(1 + \frac{1}{n} +R\sum_{D \subset U} \chi_D(z)  \big) \qquad \text{if} \quad z \in \D \setminus N(\partial U)
\end{equation}
By dividing (\ref{eqn:special_hamiltonian_lifted_differential_c}) by $k$ and taking the limit as $k \to \infty$, we acquire the first formula (\ref{eqn:special_hamiltonian_rotation_a}).

\vspace{3pt}

Next, we examine the rotation number in the region $N(\partial D)$. Fix a component disk $D \subset U$ centered at $p$ and a point $z \in N(\partial D)$. Let $S \subset N(\partial D)$ be a circle centered at $p$ with $z \in S$, and let $u \in T_zS$ be a unit tangent vector to $S$ at $z$. Finally, let
\[
S_i = \phi^i(S) \qquad z_i = \phi^i(z) \qquad w_i = \phi^G \circ \phi^i(z) \qquad u_i = \Phi(i,z)u \qquad  v_i = \Phi^G(1,\phi^i(z))\Phi(i,z)u
\]
Note that these points and vectors satisfy $z_i \in S_i$, $w_i \in S_i$, $u_i \in T_{z_i}S_i$ and $v_i \in T_{w_i}S_i$ for each $i$. By applying the decomposition (\ref{eqn:special_hamiltonian_lifted_differential}) and the additivity property (\ref{eqn:rho_s_additivity_property}) of $\rho_s$, we see that 
\begin{equation} \label{eqn:special_hamiltonian_rotation_nbhd_a}
\rho_u(\widetilde{\Phi}(k,z)) = \sum_{i=0}^{k-1} \rho_{u_i}(\widetilde{\Phi}^G(1,z_i))) + \sum_{i=0}^{k-1} \rho_{v_i}(\widetilde{\Phi}^H(1,w_i))
\end{equation}
Since $\phi^H$ is just an orthogonal rotation, we can use (\ref{eqn:special_hamiltonian_lifted_differential_a}) to immediately conclude that
\begin{equation} \label{eqn:special_hamiltonian_rotation_nbhd_b}
\rho_{v_i}(\widetilde{\Phi}^H(1,z_i))) = 1 + \frac{1}{n}\end{equation}
On the other hand, since $u_i$ is tangent to the circle $S_i$, we may use the formula (\ref{eqn:radial_action_function_1a}) to see that
\begin{equation}\label{eqn:special_hamiltonian_rotation_nbhd_c} \rho_{u_i}(\widetilde{\Phi}^G(1,z_i))) = -\frac{g_s'(r_p(z))}{2\pi r_p(z)} \end{equation}
Here $g_s$ is the function such that $G|_{N(D)} = g_s \circ r_p$.
By our hypotheses, we know that
\[
\frac{|g_s'(r_p(z))|}{2\pi r_p(z)} \le |R| \qquad \text{and} \qquad \operatorname{sgn} (-\frac{g_s'(r_p(z))}{2\pi r_p(z)}) = \operatorname{sgn}(R)
\]
By plugging in the formulas (\ref{eqn:special_hamiltonian_rotation_nbhd_a}) and (\ref{eqn:special_hamiltonian_rotation_nbhd_b}), we can estimate $\rho_u \circ \widetilde{\Phi}(k,z)$ as follows.
\[ k \cdot (1+\frac{1}{n} + \min{0,R}) \le \rho_u \circ \widetilde{\Phi}(k,z) \le k \cdot (1+\frac{1}{n} + \max{0,R}) \]
We can therefore estimate $r_\phi$. Since $\rho_u$ and $\rho$ are equivalent (Lemma \ref{lem:rhos_all_equivalent}) we find that
\[
r_\phi(z) = \lim_{k \to \infty} \frac{\rho_u \circ \widetilde{\Phi}(k,z)}{k} \qquad\text{and thus}\qquad 1+\frac{1}{n} + \min{0,R} \leq r_\phi(z) \leq 1+\frac{1}{n} + \max{0,R}
\]
Finally, the Ruelle invariant agrees with the integral of (\ref{eqn:special_hamiltonian_rotation_a}) over $\D \setminus N(\partial U)$ up to an $O(|R|\cdot \area(N(\partial U)))$ term. This proves (\ref{eqn:special_hamiltonian_rotation_c}).
\end{proof}

\begin{lemma}[Periodic Points of $\phi$] \label{lem:special_hamiltonian_periodic_points} Suppose that $R>-2$. Then the periodic points of $\phi:\D \to \D$ satisfy
\begin{equation}
\mathcal{A}(p) \ge \pi \qquad \text{and}\qquad \rho(p) + \mathcal{L}(p) > 1
\end{equation}
\end{lemma}

\begin{proof} First, consider the center $c = 0 \in \D$, where $\phi = \phi^H$. This periodic point has period $\mathcal{L}(c) = 1$. Thus, due to Lemmas \ref{lem:action_of_phi} and \ref{lem:rotation_of_phi}, the action and rotation number are given by
\[\mathcal{A}(c) = \sigma_\phi(c) = \pi(1 + \frac{1}{n}) \qquad \rho(c) =  r_\phi(c) = 1 + \frac{1}{n}\]
Any other periodic point $p$ of $H$ has period $\mathcal{L}(p) \ge n$, since $\phi$ rotates the sector $\mathbb{S}(n,k)$ to the section $\mathbb{S}(n,k+1)$. Since $n \ge 10$ and $R>-2$, we have $\sigma_\phi \ge \pi/10$ (by Lemma \ref{lem:action_of_phi}). The action of $p$ is bounded below as follows.
\[\mathcal{A}(p) = \sum_{i=0}^{\mathcal{L}(p)-1} \sigma_\phi(\phi^i(p)) \ge \frac{\pi}{10} \cdot \mathcal{L}(p) \ge \pi\]
Likewise, we apply Lemma \ref{lem:rotation_of_phi} to see that the rotation number of $p$ is bounded as follows.
\[\rho(p) = \mathcal{L}(p) \cdot r_\phi(p) \geq \mathcal{L}(p) \cdot (1 + \frac{1}{n} +\min{0,R}) > \mathcal{L}(p)\cdot (-1+\frac{1}{n}) \ge -\mathcal{L}(p) + 1 \]
In particular, the rotation number satisfies $\rho(p) + \mathcal{L}(p) > 1$.
\end{proof}

\subsection{Main Construction} \label{subsec:main_construction} We conclude this section by proving Proposition \ref{prop:dyn_cvx_counterexample}. 

\begin{proof} (Proposition \ref{prop:dyn_cvx_counterexample}) We construct the small Ruelle invariant and large Ruelle invariant contact forms separately. The basic strategy in both cases is to construct the special Hamiltonian flow of \S \ref{subsec:special_hamiltonian_map} with a specific choice of parameters, apply the open book construction of Proposition \ref{prop:open_book} to acquire a contact form and verify the desired properties by computing the relevant invariants (e.g. period, index, Calabi invariant) of $\phi$. 

\vspace{3pt}

{\bf Small Ruelle Case.} We begin by choosing the parameters $n$, $U$, $\delta$ and $R$ from Setup \ref{set:special_hamiltonian}. Fix a large positive real number $\kappa$. Choose an integer $n > \kappa$ and a union $U$ of disks $D$ that each satisfy
\begin{equation*}
\pi-\frac{1}{\kappa} < \sum\limits_{D\subset U}\area(D) < \pi\qquad \area(D)<\frac{1}{\kappa} \qquad \diam(D)<\frac{1}{\kappa}
\end{equation*}
Choose $\delta > 0$ so that $\area (N(\partial U)) < \frac{1}{\kappa}$ and choose $R := -2+\frac{1}{\kappa}$. These parameters define a special Hamiltonian flow $\phi = \phi^G \circ \phi^H$. By direct calculation and Lemma \ref{lem:action_of_phi}, we know that
\[G \# H = \pi(1 + \frac{1}{n}) \cdot (1 - r^2) \text{ near }\partial \D \qquad\text{and}\qquad \sigma_\phi > 0\]
Therefore, we can apply Construction \ref{con:open_book} to $\phi$ to acquire a contact form $\alpha$ on $S^3$.

\vspace{3pt}

We now show that (a scaling of) $\alpha$ has the desired properties. First, by Proposition \ref{prop:open_book}(b) and Lemma \ref{lem:action_of_phi}, the volume of $(S^3,\alpha)$ is given by the formula
\[\Cal(\D,\phi) = \pi^2(1 + \frac{1}{n}) +R\sum_{D \subset U } \area(D)^2 + O(d(U)) + O(R\cdot \area(N(\partial U))) = \pi^2 + O(\kappa^{-1}) \]
Next, by Proposition \ref{prop:open_book}(c) and Lemma \ref{lem:rotation_of_phi}, the Ruelle invariant of $(S^3,\alpha)$ is given by the formula
\[\Ru(\D,\phi) + \pi = \pi(2 + \frac{1}{n}) + R \sum_{D \subset U} \area(D) + O(R\cdot \area(N(\partial U)) = O(\kappa^{-1})\]
Last, by Proposition \ref{prop:open_book}(d) the binding $b = \iota(\partial \D)$ in $S^3$ has action and rotation number given by
\[\mathcal{A}(b) = \pi \qquad \rho(b) = 1 + \frac{1}{1 + 1/n} > 1\]
Due to Proposition \ref{prop:open_book}(e) and Lemma \ref{lem:special_hamiltonian_periodic_points}, every periodic orbit of $(S^3,\alpha)$ other than $b$ satisfies
\[\mathcal{A}(\gamma) \ge \pi \qquad \rho(\gamma) > 1\]
In particular, $\alpha$ is a dynamically convex contact form. Finally, rescale $\alpha$ by $\on{vol}(S^3,\alpha)^{-1/2}$. Then for any $\epsilon > 0$, we may choose a $\kappa$ sufficiently large so that
\[
\on{vol}(S^3,\alpha) = 1 \qquad \on{Ru}(S^3,\alpha) = O(\kappa^{-1}) < \epsilon \qquad \on{sys}(S^3,\alpha) > \frac{\pi^2}{\pi^2 + O(\kappa^{-1})} > 1 - \epsilon \qquad 
\]
This is precisely the list of properties \eqref{eqn:dyn_cvx_counterexample_small_ruelle}, and so we have constructed the desired small Ruelle invariant contact form.

\vspace{3pt}

{\bf Large Ruelle Case.} Again, we choose parameters $n$, $U$, $\delta$ and $R$ from Setup \ref{set:special_hamiltonian}. Fix a large positive real number $\kappa$. Choose an integer $n > \kappa$ and a union $U$ of disks $D$ that each satisfy
\begin{equation*}
\pi-\frac{1}{\kappa} < \sum\limits_{D\subset U}\area(D) < \pi\qquad \area(D)<\frac{1}{\kappa^2} \qquad \diam(D)<\frac{1}{\kappa}
\end{equation*}
Choose $\delta>0$ such that $\area(N(\partial U))<\frac{1}{\kappa^2}$ and set $R=\kappa$. These parameters define a special Hamiltonian flow $\psi = \psi^G \circ \psi^H$. By direct calculation and Lemma \ref{lem:action_of_phi}, we know that
\[G \# H = \pi(1 + \frac{1}{n}) \cdot (1 - r^2) \text{ near }\partial \D \qquad\text{and}\qquad \sigma_\phi > 0\]
Therefore, we can apply Construction \ref{con:open_book} to $\psi$ to acquire a contact form $\beta$ on $S^3$. 

\vspace{3pt}

Now we show that (a scaling of) $\beta$ has all of the desired properties. First, by Proposition \ref{prop:open_book}(b) and Lemma \ref{lem:action_of_phi}, the volume of $(S^3,\beta)$ is given by the formula
\[\Cal(\D,\psi) = \pi^2(1 + \frac{1}{n}) +R\sum_{D \subset U } \area(D)^2 + O(d(U)) + O(R\cdot \area(N(\partial U))) = \pi^2 + O(\kappa^{-1})\]
Next, by Proposition \ref{prop:open_book}(c) and Lemma \ref{lem:rotation_of_phi}, the Ruelle invariant of $(S^3,\beta)$ is given by the formula
\[\Ru(\D,\phi) + \pi = \pi(2 + \frac{1}{n}) + R \sum_{D \subset U} \area(D) + O(R\cdot \area(N(\partial U)) = \pi \cdot \kappa + O(1)\]
Last, by Proposition \ref{prop:open_book}(d) the binding $b = \iota(\partial \D)$ in $S^3$ has action and rotation number given by
\[\mathcal{A}(b) = \pi \qquad \rho(b) = 1 + \frac{1}{1 + 1/n} > 1\]
Due to Proposition \ref{prop:open_book}(e) and Lemma \ref{lem:special_hamiltonian_periodic_points}, every periodic orbit of $(S^3,\beta)$ other than $b$ satisfies
\[\mathcal{A}(\gamma) \ge \pi \qquad \rho(\gamma) > 1\]
In particular, $\beta$ is a dynamically convex contact form. Finally, rescale $\beta$ by $\on{vol}(S^3,\beta)^{-1/2}$. Then for any $\epsilon > 0$, we may choose a $\kappa$ sufficiently large so that
\[
\on{vol}(S^3,\alpha) = 1 \qquad \on{Ru}(S^3,\alpha) = O(\kappa) > \epsilon^{-1} \qquad \on{sys}(S^3,\alpha) > \frac{\pi^2}{\pi^2 + O(\kappa^{-1})} > 1 - \epsilon \qquad 
\]
This is precisely the list of properties \eqref{eqn:dyn_cvx_counterexample_large_ruelle}, and so we have constructed the desired large Ruelle invariant contact form.\end{proof}

\bibliographystyle{hplain}
\bibliography{dyn_cvx_bib}

\end{document}